\documentclass[11pt,a4paper]{article}

\usepackage[utf8]{inputenc}
\usepackage[T1]{fontenc}
\usepackage{amsmath,amssymb,amsthm}
\usepackage{mathtools}
\usepackage{hyperref}
\usepackage[margin=1in]{geometry}
\usepackage{enumitem}
\usepackage{xcolor}

\hypersetup{colorlinks=true,linkcolor=blue!60!black,urlcolor=blue!60!black,citecolor=blue!60!black}
\theoremstyle{plain}
\newtheorem{theorem}{Theorem}[section]
\newtheorem{lemma}[theorem]{Lemma}
\newtheorem{proposition}[theorem]{Proposition}
\newtheorem{corollary}[theorem]{Corollary}

\theoremstyle{definition}
\newtheorem{definition}[theorem]{Definition}
\newtheorem{remark}[theorem]{Remark}
\newtheorem{postulate}[theorem]{Postulate}

\newcommand{\R}{\mathbb{R}}
\newcommand{\Rplus}{\mathbb{R}_{>0}}
\newcommand{\Jcost}{J}
\newcommand{\Jlog}{\widetilde{J}}
\newcommand{\Kkin}{K}
\newcommand{\actionJ}{\mathcal{J}}
\newcommand{\actionA}{\mathcal{A}}
\newcommand{\actionL}{\mathcal{L}}
\newcommand{\Ehess}{\mathcal{E}_{\mathrm{Hess}}}
\newcommand{\vsr}{v_{\mathrm{SR}}}
\newcommand{\Z}{\mathbb{Z}} 

\title{\textbf{d'Alembert's Functional Equation and a Globally Convex Free-Action Principle on Positive Paths}}

\author{Sebastian Pardo-Guerra and Jonathan Washburn\\
  Recognition Physics Institute\\
  Austin, Texas, USA\\
  \texttt{sebas@recognitionphysics.org}}

\date{April 2026}

\begin{document}
\maketitle

\begin{abstract}
\noindent We study the kinetic action that d'Alembert's functional equation induces on positive paths in $\Rplus$, and prove that it is strongly convex. Calibrated d'Alembert forces the cosh cost $\Jcost(x)=\tfrac12(x+x^{-1})-1$, equivalently $\Jlog(\xi)=\cosh\xi-1$ in the log coordinate $\xi=\log x$. Evaluating this log-cost at the log-\emph{velocity} $\dot\xi$ rather than the log-position -- a single named modeling postulate (Postulate~\ref{post:step}) -- yields the kinetic action $\actionA[\gamma]=\int_a^b(\cosh\dot\xi-1)\,dt$, which is strongly convex under geometric (log-space) interpolation of paths. This convexity has three consequences, none of which requires an Euler--Lagrange equation, a Fr\'echet derivative, or a second variation. First, a one-sided chord condition characterizes global minimality. Second, the unique fixed-endpoint minimizer is the uniform-log-velocity path. Third, the action gap obeys an exact Bregman / Pythagorean identity $\actionA[\gamma]-\actionA[\gamma_*]=\int D_\Kkin(\dot\xi\,\|\,\dot\xi_*)\,dt$, sharpened by a quantitative Friedrichs--Poincar\'e bound on $\log(\gamma/\gamma_*)$. The whole package carries a dually-flat / Hessian-manifold reading in the
additive coordinate $\xi$. 
\\ This convexity theorem is purely mathematical, and we are careful to delimit it. The bridge to Newtonian and rapidity mechanics is \emph{conditional}, requiring structure beyond Postulate~\ref{post:step}: a kinematic embedding, a mass coupling, a time calibration, and a Hamiltonian-primary Legendre structure. Granted these, the cosh action recovers the Newtonian small-step limit and the relativistic rapidity profile $\Kkin_m(\phi)=m(\gamma_L-1)$; even so, the cosh-dual Hamiltonian is \emph{not} the special-relativistic free-particle Hamiltonian (Proposition~\ref{prop:not-SR}), the agreement being one of profile in rapidity rather than an identity of Hamiltonians. Global minimality, finally, is a free-sector phenomenon: once a non-affine strictly convex potential is added, joint convexity is lost and the classical local-minimum / stationary-action picture returns.

\vspace{0.5cm} 
\end{abstract}

\newpage

\section{Introduction}

The principle of least action is normally introduced as a variational
postulate~\cite{arnold1989}. In this work, we isolate a narrower mathematical
question: can a functional equation, once paired with a single
explicit step-evaluation postulate, select a kinetic action whose
fixed-endpoint minimizer is global rather than merely stationary?
We answer this question for positive paths. Throughout, we take the
choice manifold to be the positive half-line
$\Rplus$, interpreted as the space of comparison ratios in the
cost-first ledger framework of
\cite{pardoguerra_ledger2026,washburn_zlatanovic2026}, and
written in logarithmic coordinate $\xi=\log x$. We do not
re-derive this choice here. The geometric and
information-theoretic content of $\Rplus$ as a Hessian manifold is
developed in the multidimensional cost-geometry paper
\cite{washburn_zlatanovic_beltracchi2026}; the cost-first ledger
interpretation of $x \in \Rplus$ as a
ratio of ledger entries (with $x = 1$ the equilibrium ``no
comparison required'' state) is the subject of
\cite{pardoguerra_ledger2026}. We take both upstream choices as
inputs to our free-sector analysis, not as consequences of it.

Our input is d'Alembert's functional equation
\begin{equation}\label{eq:dAlembert}
  H(t+u) + H(t-u) = 2\,H(t)\,H(u), \qquad t, u \in \R.
\end{equation}
With the calibration $H(0)=1$ and $H''(0)=1$, the continuous
classification gives $H(t)=\cosh t$~\cite{aczel1966,stetkaer2013}. Thus the induced cost on
$\Rplus$ is
\[
  \Jcost(x)=\cosh(\log x)-1
    =\tfrac12(x+x^{-1})-1.
\]
The dynamical content of the work then rests on a single modeling
choice, formalized below as Postulate~\ref{post:step}: the
d'Alembert log-cost $\Jlog$ is evaluated at the log-\emph{velocity}
$\dot\xi$, not at the log-position $\xi$. This is a modeling
postulate, not a theorem of d'Alembert; \eqref{eq:dAlembert} fixes
the \emph{function} $\Jlog$ but says nothing about its argument. The
motivation for the postulate is structural: d'Alembert's equation
is a composition law for additive steps (rapidity addition,
Remark~\ref{rem:two-velocities}), so the natural object the cost
applies to is the infinitesimal log-step
$\dot\xi = d(\log\gamma)/dt$. Applying $\Jlog$ to that step gives
the kinetic integrand
\[
  \Kkin(\dot\xi)=\Jlog(\dot\xi)=\cosh(\dot\xi)-1.
\]

Our main result is then a convexity theorem. On the class of
kinetically admissible positive paths, the kinetic action
\[
  \actionA[\gamma]=\int_a^b \Kkin(\dot\xi(t))\,dt,
  \qquad \xi=\log\gamma,
\]
is \emph{strongly} convex under geometric, equivalently log-space,
interpolation of paths, with explicit $L^2$ slack
$\tfrac{s(1-s)}{2}\|\dot\xi_1-\dot\xi_2\|_{L^2}^2$. Consequently, a
one-sided chord condition implies global minimality without invoking
Euler--Lagrange equations, Fr\'echet derivatives, or second
variations. More concretely, for fixed endpoints $x_a,x_b>0$ the
unique global minimizer is
\[
  \gamma_*(t)
    =\exp\!\left(\log x_a
      +\frac{t-a}{b-a}(\log x_b-\log x_a)\right),
\]
the uniform-log-velocity path. The action gap admits an exact
Bregman / Pythagorean decomposition~\cite{bregman1967}
\[
  \actionA[\gamma] - \actionA[\gamma_*]
    \;=\; \int_a^b D_\Kkin\!\bigl(\dot\xi(t)\,\|\,\dot\xi_*\bigr)\,dt
    \;\ge\; \tfrac{1}{2}\|\dot\xi - \dot\xi_*\|_{L^2}^2
    \;\ge\; \frac{\pi^2}{2(b-a)^2}\,\bigl\|\log(\gamma/\gamma_*)\bigr\|_{L^2}^2,
\]
where $D_\Kkin(v\|w)=\cosh v - \cosh w - \sinh w\,(v - w)$ is the
cosh Bregman divergence. The minimum-action profile
$\actionA_*(T,\Delta)=T(\cosh(\Delta/T)-1)$ is jointly convex and
positively $1$-homogeneous, hence subadditive in time. These
quantitative refinements admit a clean dually-flat / Hessian-manifold
reading~\cite{amari_nagaoka2000,shima2007} in the additive coordinate $\xi=\log x$: the log-cost
$\Jlog(\xi)=\cosh\xi-1$ is the Hessian potential of a 1D dually-flat
structure with metric $\cosh\xi\,d\xi^2$, whose $e$-affine
(log-affine) geodesics are exactly the geometric interpolations, and
$\actionA$ is the corresponding Bregman (cosh-)energy. This
log-coordinate structure is distinct from the Hessian metric
$g_J=x^{-3}\,dx^2$ of $\Jcost$ in the coordinate $x$
(Appendix~\ref{app:hessian}).

We confine this global-minimizer statement to the free sector. Once
a non-affine strictly convex potential is added, joint convexity of
the Lagrangian is lost and the classical local-minimum /
stationary-action picture reappears. We therefore make the later
bridge to Newtonian mechanics explicitly conditional: it requires a
kinematic embedding, a mass coupling, a time calibration, and a
Hamiltonian-primary Legendre structure. These choices explain how the cosh kinetic cost
has the Newtonian small-step limit and the relativistic rapidity
profile, but they are not part of the free convexity theorem itself.

We retain the velocity-free integral
$\actionJ[\gamma]=\int\Jcost(\gamma(t))\,dt$ only as a static cost
(Section~\ref{sec:pathspace}); its integrand has no velocity
dependence and its Euler--Lagrange equation is algebraic, selecting
the ground state $\gamma\equiv 1$ (Section~\ref{sec:el}). We defer
the comparison of $\actionA$ with the Hessian Riemannian
path-energy on $(\Rplus,g_J)$
to Appendix~\ref{app:hessian}.

\medskip
\noindent\emph{Nature of the contribution.} The free-sector
arguments are deliberately elementary: once d'Alembert fixes the
cosh cost and Postulate~\ref{post:step} places it at the
log-velocity, the convexity results follow from pointwise convexity
of $\cosh$ after the log change of coordinates, Jensen's
inequality, a perspective construction~\cite{rockafellar1970,boyd_vandenberghe2004},
the one-dimensional Friedrichs inequality~\cite{hardy_littlewood_polya1952}, and
textbook Bregman / dually-flat geometry~\cite{bregman1967,amari_nagaoka2000,shima2007}.
The novelty we claim is therefore one of \emph{provenance and
organization}, not of analytic depth: a functional equation together
with a single named postulate forces a globally (not merely locally)
minimizing free action with an exact Bregman gap and a closed-form
geodesic minimizer. The fuller scope statement, including the
conditional status of the mechanics bridge, is given in
Section~\ref{sec:discussion}.

\subsection{Structure of the paper}

The paper has two parts. The free-sector mathematical core occupies
Sections~\ref{sec:cost}--\ref{sec:el}, and the conditional physical
bridge occupies Sections~\ref{sec:newton}--\ref{sec:scope}.

\begin{itemize}[noitemsep,topsep=2pt]
  \item \S\ref{sec:cost} recalls the calibrated d'Alembert
    classification and introduces the costs $\Jcost,\Jlog,\Kkin$.
  \item \S\ref{sec:pathspace} sets up positive path spaces,
    $\actionA$, and the geometric and arithmetic interpolations.
  \item \S\ref{sec:convexity} proves the strong convexity theorem,
    the fixed-endpoint minimizer, the Pythagorean / Bregman identity
    with its Friedrichs--Poincar\'e bound, the perspective convexity
    of the minimum-action profile, and the dually flat reformulation.
  \item \S\ref{sec:el} compares the static, kinetic, and Hessian
    Euler--Lagrange pictures.
  \item \S\ref{sec:newton} derives Newton's law from the cosh
    Lagrangian.
  \item \S\ref{sec:hamilton} gives the Hamiltonian formulation and
    the Noether conservation laws.
  \item \S\ref{sec:scope} records the classical local-min /
    stationary-action picture for general potentials.
\end{itemize}

\noindent Section~\ref{sec:discussion} summarizes the results and
lists open directions, and Appendix~\ref{app:hessian} collects the
comparison material on the Hessian Riemannian path-energy $\Ehess$.

\begin{remark}[Scope of the framework: what is and is not claimed]\label{rem:scope}
Because the framework draws on several traditions
(d'Alembert's functional equation, Bregman / dually flat geometry,
Newtonian mechanics, special relativity), it is worth stating up
front what each piece does and does not assert.

\begin{itemize}
  \item \emph{Two velocity notions.} Under the kinematic embedding
  of Definition~\ref{def:embedding}, the coordinate velocity
  $\dot q$ equals the boost rapidity $\phi$ (in natural units),
  \emph{not} the SR $3$-velocity $\vsr = c\tanh\phi$. The two
  agree only at leading order, $\phi = \vsr/c + O((\vsr/c)^3)$;
  for large rapidity $\vsr$ saturates at $\pm c$ while
  $\dot q = \phi$ remains unbounded.

  \item \emph{Cosh-dual Hamiltonian is not the SR Hamiltonian.}
  The Hamiltonian $T_H(p) =
  p\operatorname{arsinh}(p/m)-\sqrt{m^2+p^2}+m$ of
  \S\ref{sec:hamilton} is the Legendre dual of the cosh kinetic
  cost. It is \emph{not} the SR free-particle Hamiltonian
  $\sqrt{m^2+p^2}-m$; the two agree only at $O(p^2)$ and differ
  strictly at $O(p^4)$, with $T_H(p) > \sqrt{m^2+p^2}-m$ for all
  $p\neq 0$ (Proposition~\ref{prop:not-SR}). The
  Proposition~\ref{prop:rapidity} match is one of \emph{profile}
  in rapidity, not Hamiltonian identity.

  \item \emph{General potentials are external.} The native
  cosh-sinh potential $V_{\mathrm{nat}} = k\Jlog$ is the only
  potential forced by the same d'Alembert uniqueness that forces
  the kinetic term. All other potentials (harmonic, Coulomb,
  gravitational, polynomial, lattice, periodic, etc.) are external
  inputs from the cost-field environment; we do not derive them
  from \eqref{eq:dAlembert}.

  \item \emph{Global minimality is a free-sector claim.} The
  chord-form free-action principle (Theorem~\ref{thm:pla}) gives
  global minimality of $\actionA$ on kinetically admissible
  positive paths sharing endpoints. Once a non-affine strictly
  convex potential is added, the Lagrangian becomes indefinite in
  $(\xi,\dot\xi)$ and the classical short-time
  local-minimum~/ long-time stationary-action~/ conjugate-time
  obstruction picture replaces global minimality
  (\S\ref{sec:scope}).
\end{itemize}
\end{remark}

\section{The Cost Functional, Log Coordinates, and the Kinetic Cost}\label{sec:cost}

This section assembles the cost functions on which the rest of the
paper is built. We first recall the calibrated d'Alembert
classification, then introduce the static cost $\Jcost$ on $\Rplus$,
its logarithmic form $\Jlog$, and the kinetic cost $\Kkin$ obtained
by evaluating $\Jlog$ at the log-velocity (Postulate~\ref{post:step}).
We begin by fixing notation.

\begin{remark}[Notation glossary]\label{rem:notation}
For convenience we list here the principal symbols used in the cost,
action, Lagrangian, and Hamiltonian layers of the paper, with a
pointer to where each is first introduced. The free-sector symbols
($\Jcost$ through $\actionA_*$) are dimensionless throughout; the
mechanics-layer symbols ($\Kkin_m$ onward) carry the dimensions
recorded in the dimensional dictionary
(Remark~\ref{rem:dim-dict}).

\smallskip
\noindent\begin{tabular}{@{}lll@{}}
\hline
Symbol & Role & First introduced \\ \hline
$\Jcost(x)$ & static cost on $\Rplus$, $\tfrac12(x+x^{-1})-1$
            & \S\ref{sec:cost} \\
$\Jlog(\xi)$ & log-cost, $\cosh\xi - 1$
             & \S\ref{sec:cost} \\
$\Kkin(v)$ & kinetic cost, $\Jlog(v) = \cosh v - 1$
            & \S\ref{sec:cost}, Postulate~\ref{post:step} \\
$\actionJ[\gamma]$ & static integral $\int_a^b \Jcost(\gamma)\,dt$
                   & \S\ref{sec:pathspace} \\
$\actionA[\gamma]$ & kinetic action
                     $\int_a^b \Kkin(\dot\xi)\,dt$
                   & \S\ref{sec:pathspace} \\
$\actionA_*(T,\Delta)$ & minimum-action profile,
                         $T(\cosh(\Delta/T) - 1)$
                       & \S\ref{subsec:perspective} \\
$\Ehess[\gamma]$ & Hessian Riemannian path-energy
                   (distinct from $\actionA$)
                 & \S\ref{subsec:ground-state},
                   App.~\ref{app:hessian} \\
$\Kkin_m(\phi)$ & physical kinetic cost,
                  $m(\cosh\phi - 1)$
                & \S\ref{sec:newton}, Def.~\ref{def:mass} \\
$\actionL(\xi,\dot\xi)$ & cosh Lagrangian,
                          $\Kkin_m(\dot\xi) - V(\xi)$
                        & \S\ref{subsec:general} \\
$L_{\mathrm{nat}}$ & native d'Alembert Lagrangian
                   & Def.~\ref{def:native} \\
$T_H(p)$ & cosh-dual Hamiltonian
           ($\neq \sqrt{m^2+p^2}-m$,
           Prop.~\ref{prop:not-SR})
         & Prop.~\ref{prop:TH} \\
$H(\xi,p)$ & Hamiltonian, $T_H(p) + V(\xi)$
           & Def.~\ref{def:additive} \\
\hline
\end{tabular}

\smallskip
\noindent The kinematic symbols ($\xi$, $\phi$, $\vsr$, $q$, $\tau$,
$t_0$) used in the mechanics layer are catalogued separately in
Remark~\ref{rem:dim-dict}. Two symbols carry conventional double
duty within standard physics usage but are kept distinct here by
notation: the Hessian path-energy is written $\Ehess$, never $E$
(reserved for the conserved energy along a trajectory in
\S\ref{sec:hamilton}); and the d'Alembert solution function $H$ of
\eqref{eq:dAlembert} (Theorem~\ref{thm:dAlembert-classification})
appears only inside \S\ref{sec:cost}, never in the same equation as
the Hamiltonian $H(\xi,p)$ of \S\ref{sec:hamilton}.
\end{remark}

\subsection{Calibrated d'Alembert classification}

\begin{theorem}[Calibrated d'Alembert classification, recalled]\label{thm:dAlembert-classification}
Let $H : \R \to \R$ be a continuous solution of the d'Alembert
functional equation \eqref{eq:dAlembert}.
\begin{enumerate}[label=(\roman*),noitemsep,topsep=2pt]
  \item \textbf{Smoothness (Acz\'el).} By Acz\'el's smoothness theorem
    for d'Alembert's equation \cite[Ch.~3, \S 3.1.3]{aczel1966},
    continuity of $H$ on $\R$ implies $H \in C^\infty(\R)$. This is
    the only external bridge result used in the classification.
  \item \textbf{Classification.} Every continuous solution of
    \eqref{eq:dAlembert} is exactly one of the following disjoint
    alternatives:
    \[
      H(t) \equiv 0, \qquad
      H(t) = \cos(\alpha t)\;(\alpha>0), \qquad
      H(t) = \cosh(\alpha t)\;(\alpha>0), \qquad
      H(t) \equiv 1.
    \]
    The zero branch occurs when $H(0)=0$. If $H(0)=1$, then
    $H$ is even, $H'(0)=0$, and differentiating
    \eqref{eq:dAlembert} twice in the second variable at $0$ yields
    $H''(t)=H''(0)H(t)$. The sign of $H''(0)$ gives the cosine,
    cosh, or constant branch.
  \item \textbf{Calibration to $\cosh$.} The condition $H(0)=1$
    excludes the zero branch, and the condition $H''(0)>0$ excludes
    the cosine and constant branches, forcing $H(t)=\cosh(\alpha t)$
    for some $\alpha>0$. Fixing the unit of cost by the further
    normalization $H''(0)=1$ then selects $\alpha=1$, hence
    $H(t)=\cosh t$.
\end{enumerate}
\end{theorem}

\begin{remark}[$H''(0)=1$ is a unit choice, not forced by d'Alembert]\label{rem:calibration-unit}
d'Alembert's equation~\eqref{eq:dAlembert} together with $H(0)=1$ and
$H''(0)>0$ forces $H(t)=\cosh(\alpha t)$ only \emph{up to} the scale
parameter $\alpha=\sqrt{H''(0)}>0$. The choice $\alpha=1$ we adopt
throughout is a normalization of the cost scale, equivalent to
choosing the unit in which the kinetic coefficient $K''(0)=\cosh
0=1$ at the ground state. Any other choice $\alpha>0$ produces the
same theory under the rescaling $\xi\mapsto\alpha\xi$, $K\mapsto
\cosh(\alpha\,\cdot)-1$. No qualitative result here depends on the
value of $\alpha$; only the numerical action constants do.
\end{remark}

\begin{remark}[Aczél's smoothness theorem, restated for self-containedness]\label{rem:aczel-restated}
For ease of reference and completeness, we restate the precise
statement of the single external bridge result on which
Theorem~\ref{thm:dAlembert-classification}(i) relies.

\smallskip
\noindent\emph{(Aczél)} \cite[Ch.~3, §3.1.3]{aczel1966}.
\emph{Let $H : \R \to \R$ be a continuous, not identically zero
solution of d'Alembert's equation \eqref{eq:dAlembert}. Then $H$ is
infinitely differentiable on $\R$.}

\smallskip
\noindent The proof in Aczél proceeds by a translation-and-averaging
argument. Setting $u=t$ in \eqref{eq:dAlembert} gives the algebraic
identity
\[
  H(2t) + H(0) \;=\; 2H(t)^2,
\]
which expresses $H$ on a dilated argument as a polynomial in
$H(t)$ and so propagates information about $H$ to all dyadic
scales; this identity by itself does not raise regularity. The
regularity bootstrap itself comes from integrating
\eqref{eq:dAlembert} in $u$ over $[0,h]$. Choosing $h>0$ small
enough that $C_h := \int_0^h H(u)\,du \ne 0$ (possible on the
nontrivial branch, where $H(0)=1$ and $H$ is continuous), the
substitutions $s = t\pm u$ give
\[
  2\,C_h\,H(t)
    \;=\; \int_t^{t+h} H(s)\,ds + \int_{t-h}^t H(s)\,ds.
\]
The right-hand side is $C^1$ in $t$ by the fundamental theorem of
calculus, hence so is $H$; iterating, if $H \in C^k$ then the
right-hand side is $C^{k+1}$, hence $H \in C^{k+1}$, so
$H \in C^\infty$. The upshot is that $H'(0)$ and $H''(0)$ both
exist as ordinary derivatives, which is exactly what the
calibration in Theorem~\ref{thm:dAlembert-classification}(iii)
requires. This is the only place in our development where we
invoke a continuity-to-smoothness upgrade from outside; everything
downstream stays inside $C^\infty$.
\end{remark}

\subsection{The static cost functional}

\begin{definition}[Cost functional]
The \emph{cost functional} is
\[
  \Jcost(x) \;=\; \tfrac{1}{2}\!\left(x + x^{-1}\right) - 1,
    \qquad x > 0.
\]
\end{definition}

$\Jcost$ is non-negative, vanishes iff $x = 1$, is symmetric under
inversion ($\Jcost(x) = \Jcost(1/x)$), and diverges at $0$ and
$\infty$.

\begin{theorem}[Strict convexity of $\Jcost$]\label{thm:Jconvex}
$\Jcost$ is strictly convex on $(0,\infty)$.
\end{theorem}

\begin{proof}
$\Jcost''(x) = x^{-3} > 0$ for $x > 0$.
\end{proof}

\subsection{Logarithmic coordinates}

Set $\xi := \log x$, so that $x = e^\xi$ and $\xi$ ranges over $\R$.

\begin{definition}[Log-cost functional]
The \emph{log-cost functional} is
\[
  \Jlog(\xi) \;:=\; \Jcost(e^\xi)
    \;=\; \tfrac12(e^\xi + e^{-\xi}) - 1
    \;=\; \cosh(\xi) - 1.
\]
\end{definition}

\begin{proposition}[Properties of $\Jlog$]\label{prop:Jlog}
\begin{enumerate}[label=(\roman*),noitemsep]
  \item $\Jlog(\xi) \geq 0$ with equality iff $\xi = 0$.
  \item $\Jlog$ is even: $\Jlog(-\xi) = \Jlog(\xi)$, which is the
    log-form of $\Jcost(x) = \Jcost(1/x)$.
  \item $\Jlog$ is strictly convex on $\R$, with
    $\Jlog''(\xi) = \cosh(\xi) \geq 1 > 0$.
  \item Small-$\xi$ Taylor:
    $\Jlog(\xi) = \tfrac12\xi^2 + \tfrac{1}{24}\xi^4 + O(\xi^6)$.
\end{enumerate}
\end{proposition}

\begin{proof}
All from $\Jlog(\xi) = \cosh(\xi) - 1$ and the standard Taylor series
of $\cosh$.
\end{proof}

\begin{remark}
The passage $\Jcost \leftrightarrow \Jlog$ via $\xi = \log x$ is the
natural one for d'Alembert. In the additive variable $\xi$,
\eqref{eq:dAlembert} is precisely the cosh-composition identity
$\cosh(\xi+\eta) + \cosh(\xi-\eta) = 2\cosh\xi\cosh\eta$, so $\Jlog$
is the d'Alembert-calibrated cost \emph{evaluated at an additive
argument}. The remainder of our analysis lives naturally in these
coordinates.
\end{remark}

\subsection{The kinetic cost}

The classification of d'Alembert solutions
(Theorem~\ref{thm:dAlembert-classification}) fixes the
\emph{function} $\Jlog(\cdot) = \cosh(\cdot) - 1$, but it does not
specify which argument that function takes on. Two natural choices
arise on a positive path $\gamma:[a,b]\to\Rplus$ with
$\xi = \log\gamma$: the log-\emph{position} $\xi(t)$ and the
log-\emph{velocity} $\dot\xi(t) = d(\log\gamma)/dt$. We make the
second choice explicit as a postulate, since it is what generates
the dynamics of this paper.

\begin{postulate}[Step-evaluation of the d'Alembert log-cost]\label{post:step}
Along a positive path $\gamma:[a,b]\to\Rplus$ with $\xi=\log\gamma$,
the d'Alembert log-cost $\Jlog$ is evaluated at the
log-\emph{velocity} $\dot\xi$, not at the log-position $\xi$. The
cost per unit parameter time of motion with log-velocity $\dot\xi$
is therefore
\[
  \Jlog(\dot\xi) \;=\; \cosh(\dot\xi) - 1.
\]
\end{postulate}

\begin{remark}[Status and consequences of Postulate~\ref{post:step}]\label{rem:postulate-status}
Postulate~\ref{post:step} is a \emph{modeling choice}, not a
consequence of d'Alembert's equation: \eqref{eq:dAlembert} fixes the
function $\Jlog$ but says nothing about its argument. Its motivation
is structural -- d'Alembert's equation is the composition law for
additive \emph{steps} (rapidity addition, in the kinematic reading
of \S\ref{sec:newton}, Remark~\ref{rem:two-velocities}), and the
canonical additive object on a positive path is the infinitesimal
log-step $\dot\xi$. We do not claim this motivation forces the
postulate.

Evaluating $\Jlog$ at the log-\emph{position} $\xi$ instead gives
the velocity-free static cost integrand of \S\ref{sec:pathspace},
whose Euler--Lagrange equation is algebraic
(Remark~\ref{rem:cost-rate-EL}) and produces no dynamics. The
strongly convex kinetic action of \S\ref{sec:convexity}, and with
it the entire free-sector global-minimum result of this paper, is
downstream of Postulate~\ref{post:step}. The same $\Jlog$ also
reappears at log-position in \S\ref{sec:newton} as the native
potential $V_{\mathrm{nat}}$ once a Hamiltonian-primary additive
cost postulate is added; until then, only the velocity reading is
in force.
\end{remark}

\begin{definition}[Kinetic cost]
Under Postulate~\ref{post:step}, the \emph{kinetic cost} is the
real-valued function $\Kkin : \R \to \R$,
\[
  \Kkin(v) \;:=\; \Jlog(v) \;=\; \cosh(v) - 1.
\]
The value $\Kkin(\dot\xi)$ is the d'Alembert-calibrated cost per
unit parameter time of motion with log-velocity $\dot\xi$. In the
physical bridge of Section~\ref{sec:newton}, this parameter is
calibrated to the dimensionless time $\tau=t_{\mathrm{phys}}/t_0$;
before that bridge, $t$ is only the variational parameter on the path.
\end{definition}

\begin{proposition}[Properties of $\Kkin$]\label{prop:Kkin}
\begin{enumerate}[label=(\roman*),noitemsep]
  \item $\Kkin(v) \geq 0$ with equality iff $v = 0$ (zero
    log-velocity).
  \item $\Kkin$ is even.
  \item $\Kkin$ is strictly convex on $\R$, with
    $\Kkin''(v) = \cosh(v) \geq 1$.
  \item \textbf{Sharp quartic remainder.} For all $|v| < 1$,
    \begin{equation}\label{eq:Kkin-sharp-quartic}
      \bigl|\Kkin(v) - \tfrac12 v^2\bigr|
      \;\leq\; \frac{v^4}{24\,(1 - v^2)}.
    \end{equation}
    In particular, $\Kkin(v) - \tfrac12 v^2 = \tfrac{v^4}{24} +
    O(v^6)$ as $v\to 0$, and the leading constant $1/24$ is sharp.
\end{enumerate}
\end{proposition}

\begin{proof}
(i)--(iii) are immediate from $\Kkin(v) = \cosh(v) - 1$ and
$\Kkin''(v) = \cosh(v) \ge 1$.

For (iv), the cosh Taylor series gives
\[
  \Kkin(v) - \tfrac12 v^2
    \;=\; \sum_{k\ge 2} \frac{v^{2k}}{(2k)!}.
\]
All terms are non-negative, so the difference is non-negative
(the absolute value can be dropped). For each $k \ge 2$,
$(2k)! \ge 4! = 24$ and $v^{2k} = v^4 \cdot (v^2)^{k-2}$, so
$v^{2k}/(2k)! \le v^4(v^2)^{k-2}/24$. Summing the resulting
geometric series in $v^2$ (valid for $|v| < 1$),
\[
  0 \;\le\; \sum_{k\ge 2} \frac{v^{2k}}{(2k)!}
  \;\le\; \frac{v^4}{24}\sum_{j\ge 0} (v^2)^j
  \;=\; \frac{v^4}{24\,(1 - v^2)},
\]
which is \eqref{eq:Kkin-sharp-quartic}. Sharpness of the leading
$1/24$ follows from the $k=2$ term: as $v\to 0$,
$\Kkin(v) - \tfrac12 v^2 \sim v^4/24$.
\end{proof}

\begin{remark}[Rapidity interpretation]
If we read the dimensionless step variable $\phi := d\xi/d\tau$ as
the rapidity of a one-parameter boost, then
$\Kkin(\phi) = \cosh(\phi) - 1$ is the excess cosh of rapidity --
the relativistic kinetic-energy profile. d'Alembert's equation
\eqref{eq:dAlembert} is the rapidity-addition law, so this
profile is internal to the d'Alembert framework. The dynamical
identification of $\phi$ with physical rapidity is the kinematic
calibration we make in Definition~\ref{def:embedding}.
\end{remark}

\section{Path Space, Static and Kinetic Actions, and Interpolation}\label{sec:pathspace}

With the cost functions $\Jcost$, $\Jlog$, and $\Kkin$ now in hand,
we turn to the spaces of paths on which they act. This section fixes
the admissibility classes used throughout the paper, defines the
static cost integral $\actionJ$ and the kinetic action $\actionA$,
and introduces the two interpolations -- arithmetic and geometric --
whose contrast drives the convexity theory of
Section~\ref{sec:convexity}.

\subsection{Admissible paths}

\begin{definition}[Admissible path]\label{def:admissible}
For $a,b \in \R$ with $a < b$, an \emph{admissible path on
$[a,b]$} is a function $\gamma : [a,b] \to \Rplus$ that is absolutely
continuous on $[a,b]$, satisfies $\gamma(t) > 0$ for every
$t \in [a,b]$, and has $\dot\gamma \in L^2([a,b])$.
\end{definition}

\begin{remark}[Automatic uniform positivity]\label{rem:auto-positive}
Absolute continuity on the compact interval $[a,b]$ implies continuity,
so the strict positivity condition $\gamma(t) > 0$ on the compact
$[a,b]$ together with continuity yields a positive lower bound: there
exists $c_\gamma > 0$ such that $\gamma(t) \ge c_\gamma$ for all
$t \in [a,b]$. We therefore use ``strictly positive on $[a,b]$'' and
``bounded away from $0$'' interchangeably for admissible paths.
\end{remark}

Strict positivity is essential: $\Jcost$ is defined only on
$(0,\infty)$, and we will need $\log\gamma$ to be absolutely
continuous with $L^2$ derivative. The log-coordinate of an admissible
path, $\xi(t) := \log\gamma(t)$, is absolutely continuous on $[a,b]$
with $\dot\xi = \dot\gamma/\gamma \in L^2$ (using
Remark~\ref{rem:auto-positive} to bound $1/\gamma$ above by
$1/c_\gamma < \infty$).

\begin{definition}[Finite-action (kinetically admissible) path]\label{def:kin-adm}
An admissible path $\gamma$ on $[a,b]$ is \emph{kinetically admissible}
(or \emph{finite-action}) if, writing $\xi := \log\gamma$, one has
$\Kkin(\dot\xi) \in L^1([a,b])$, equivalently
\[
  \int_a^b \bigl[\cosh(\dot\xi(t)) - 1\bigr]\,dt < \infty.
\]
\end{definition}

\begin{remark}\label{rem:actionA-wellposed}
The condition $\dot\xi \in L^2([a,b])$ does \emph{not} by itself imply
$\Kkin(\dot\xi) \in L^1([a,b])$, since $\cosh$ has super-quadratic
growth. We therefore treat the kinetic action $\actionA$ below as a
real-valued functional on the class of kinetically admissible paths.
\end{remark}

\subsection{The static cost integral}

\begin{definition}[Static cost integral]
For an admissible path $\gamma$ on $[a,b]$, the \emph{static cost
integral} is
\[
  \actionJ[\gamma] \;:=\; \int_a^b \Jcost(\gamma(t))\,dt.
\]
\end{definition}

\begin{proposition}[Basic properties of $\actionJ$]\label{prop:S-basic}
$\actionJ[\gamma] \geq 0$, and
$\actionJ[\gamma_{\mathrm{const}=1}] = 0$.
\end{proposition}

\begin{proof}
An admissible path is continuous on the compact interval $[a,b]$ and
takes values in $(0,\infty)$, hence $\Jcost\circ\gamma$ is continuous
and therefore integrable. Since $\Jcost(x)\ge0$ for all $x>0$, its
integral is non-negative. If $\gamma\equiv1$, then
$\Jcost(\gamma(t))=\Jcost(1)=0$ for all $t$, so the integral is zero.
\end{proof}

\noindent $\actionJ$ measures the time-integrated cost of
\emph{being in} a non-ground state and plays a role analogous to a
potential term in the multiplicative coordinate. It is not an action
in the dynamical sense: its integrand has no velocity dependence and
its Euler--Lagrange equation is algebraic (Section~\ref{sec:el}).

\subsection{The kinetic action}

\begin{definition}[Kinetic action]
For a kinetically admissible path $\gamma$ on $[a,b]$, with
$\xi := \log\gamma$, the \emph{kinetic action} is
\[
  \actionA[\gamma]
    \;:=\; \int_a^b \Kkin(\dot\xi(t))\,dt
    \;=\; \int_a^b \bigl[\cosh(\dot\xi(t)) - 1\bigr]\,dt.
\]
\end{definition}

\begin{proposition}[Basic properties of $\actionA$]\label{prop:A-basic}
\begin{enumerate}[label=(\roman*),noitemsep]
  \item $\actionA[\gamma] \geq 0$.
  \item $\actionA[\gamma] = 0$ iff $\gamma$ is constant on $[a,b]$
    (any positive constant; not only $\gamma \equiv 1$).
  \item If $\|\dot\xi\|_\infty \leq 1/10$, then the kinetic action
    approximates the Newtonian $L^2$ integral with relative error
    at most $1/1188$:
    \[
      \bigl|\actionA[\gamma] - \tfrac12 \textstyle\int_a^b \dot\xi(t)^2\,dt\bigr|
      \;\le\; \tfrac{100}{99\cdot 24}\textstyle\int_a^b \dot\xi(t)^4\,dt
      \;\le\; \tfrac{1}{2376}\textstyle\int_a^b \dot\xi(t)^2\,dt.
    \]
\end{enumerate}
\end{proposition}

\begin{proof}
(i) $\Kkin \geq 0$ pointwise. (ii) $\Kkin(\dot\xi) = 0$ a.e.\
$\iff \dot\xi = 0$ a.e., and absolute continuity of $\xi$ gives
$\xi$ constant, equivalently $\gamma$ constant. (iii) Apply the
sharp quartic remainder \eqref{eq:Kkin-sharp-quartic} pointwise at
$v = \dot\xi(t)$ with $|v|\le 1/10$, so that $(1-v^2)^{-1}\le
100/99$ and $v^4\le v^2/100$; integration yields both bounds.
\end{proof}

\subsection{Two interpolations}

The variational theorems below require an interpolation between two
admissible paths. Two natural candidates arise: \emph{arithmetic}
interpolation (appropriate for $\actionJ$) and \emph{geometric} (log-
space) interpolation (appropriate for $\actionA$).

\begin{definition}[Arithmetic interpolation]
For admissible $\gamma_1, \gamma_2$ on $[a,b]$ and $s \in [0,1]$,
\[
  \mathrm{interp}^+(\gamma_1, \gamma_2, s)(t)
    \;:=\; (1-s)\,\gamma_1(t) + s\,\gamma_2(t).
\]
\end{definition}

\begin{definition}[Geometric (log-space) interpolation]
For admissible $\gamma_1, \gamma_2$ on $[a,b]$ and $s \in [0,1]$,
\[
  \mathrm{interp}^\times(\gamma_1, \gamma_2, s)(t)
    \;:=\; \gamma_1(t)^{1-s}\,\gamma_2(t)^s
    \;=\; \exp\!\bigl((1-s)\log\gamma_1(t) + s\log\gamma_2(t)\bigr).
\]
Equivalently, in log coordinates,
$\xi_s(t) = (1-s)\xi_1(t) + s\xi_2(t)$.
\end{definition}

\begin{lemma}[Both interpolations preserve admissibility]\label{lem:interp-pos}
If $\gamma_1, \gamma_2$ are admissible on $[a,b]$ and $s \in [0,1]$,
then $\mathrm{interp}^+(\gamma_1, \gamma_2, s)$ and
$\mathrm{interp}^\times(\gamma_1, \gamma_2, s)$ are admissible on
$[a,b]$.
\end{lemma}

\begin{proof}
For $\mathrm{interp}^+$, absolute continuity and the $L^2$ derivative
condition are preserved under linear combinations, and positivity holds
since $(1-s)\gamma_1 + s\gamma_2 \geq \min(\gamma_1,\gamma_2) > 0$.

For $\mathrm{interp}^\times$, write $\xi_i := \log\gamma_i$. Since each
$\gamma_i$ is bounded away from $0$ and absolutely continuous, $\xi_i$
is absolutely continuous with $\dot\xi_i = \dot\gamma_i/\gamma_i \in
L^2([a,b])$. Then $\xi_s := (1-s)\xi_1 + s\xi_2$ is absolutely
continuous with $\dot\xi_s \in L^2$, and
$\gamma_s := \mathrm{interp}^\times(\gamma_1,\gamma_2,s) = e^{\xi_s}$
is absolutely continuous with $\dot\gamma_s = \gamma_s\,\dot\xi_s$ a.e.
Since $\gamma_1,\gamma_2$ are bounded away from $0$, there exist
$c_1,c_2>0$ with $\gamma_i(t)\ge c_i$ on $[a,b]$, hence
$\gamma_s(t)=\gamma_1(t)^{1-s}\gamma_2(t)^s\ge c_1^{1-s}c_2^s>0$, so
$\gamma_s$ is bounded away from $0$. On the compact interval $[a,b]$,
$\gamma_s$ is continuous and strictly positive, hence also bounded
above, so $\dot\gamma_s=\gamma_s\dot\xi_s\in L^2([a,b])$.
\end{proof}

\begin{lemma}[Geometric interpolation preserves finite kinetic action]\label{lem:interp-kin}
If $\gamma_1,\gamma_2$ are kinetically admissible on $[a,b]$ and
$s\in[0,1]$, then $\mathrm{interp}^\times(\gamma_1,\gamma_2,s)$ is
kinetically admissible on $[a,b]$.
\end{lemma}

\begin{proof}
Let $\xi_i=\log\gamma_i$ and $\xi_s=(1-s)\xi_1+s\xi_2$, so
$\dot\xi_s=(1-s)\dot\xi_1+s\dot\xi_2$ a.e. By convexity of $\Kkin$,
\[
  \Kkin(\dot\xi_s(t))
    \le (1-s)\Kkin(\dot\xi_1(t)) + s\Kkin(\dot\xi_2(t))
  \qquad\text{for a.e.\ }t.
\]
The right-hand side lies in $L^1([a,b])$ by kinetic admissibility of
$\gamma_1,\gamma_2$, hence so does $\Kkin(\dot\xi_s)$.
\end{proof}

\section{Convexity and the Free-Action Principle}\label{sec:convexity}

We now reach the mathematical core of the paper. Having assembled the
cost functions and the path spaces, we establish the convexity
properties of the static and kinetic actions, derive the free-action
principle in chord form, and extract its quantitative consequences:
the closed-form fixed-endpoint minimizer, the Bregman--Pythagorean
identity with its Friedrichs--Poincar\'e bound, the perspective
convexity of the minimum-action profile, and the dually flat
reformulation.

\subsection{Static convexity}

\begin{theorem}[Convexity of $\actionJ$ under arithmetic interpolation]\label{thm:actionJ-convex}
For admissible $\gamma_1, \gamma_2$ on $[a,b]$ and $s \in [0,1]$,
\[
  \actionJ[\mathrm{interp}^+(\gamma_1, \gamma_2, s)]
    \;\leq\; (1-s)\,\actionJ[\gamma_1] + s\,\actionJ[\gamma_2].
\]
\end{theorem}

\begin{proof}
Pointwise convexity of $\Jcost$
(Theorem~\ref{thm:Jconvex}): at every $t$,
$\Jcost((1-s)\gamma_1(t) + s\gamma_2(t)) \leq
(1-s)\Jcost(\gamma_1(t)) + s\Jcost(\gamma_2(t))$. Integrate.
\end{proof}

\subsection{Kinetic convexity (the central theorem)}

The kinetic action $\actionA$ is \emph{not} in general convex under
arithmetic interpolation in $\gamma$. The correct interpolation is
geometric: in log coordinates, this is affine, and convexity of the
integrand $\Kkin(\dot\xi)$ in $\dot\xi$ propagates immediately.

\begin{theorem}[Strong convexity of $\actionA$ under geometric interpolation]\label{thm:actionA-convex}
For kinetically admissible $\gamma_1, \gamma_2$ on $[a,b]$, write
$\xi_i := \log\gamma_i$. For every $s \in [0,1]$,
\begin{equation}\label{eq:strong-convex-A}
  \actionA[\mathrm{interp}^\times(\gamma_1, \gamma_2, s)]
    \;\leq\; (1-s)\,\actionA[\gamma_1] + s\,\actionA[\gamma_2]
    \;-\; \frac{s(1-s)}{2}\int_a^b \bigl(\dot\xi_1(t) - \dot\xi_2(t)\bigr)^2\,dt.
\end{equation}
In particular, $\actionA$ is convex along geometric interpolation,
and inequality \eqref{eq:strong-convex-A} is strict for
$s \in (0,1)$ unless $\dot\xi_1 = \dot\xi_2$ a.e.\ on $[a,b]$;
equivalently, $\gamma_2/\gamma_1$ is constant on $[a,b]$.
\end{theorem}

\begin{proof}
Set $\xi_s := (1-s)\xi_1 + s\xi_2$, so $\xi_s$ is the log of
$\mathrm{interp}^\times(\gamma_1, \gamma_2, s)$ and
$\dot\xi_s = (1-s)\dot\xi_1 + s\dot\xi_2$ a.e. The kinetic
integrand $\Kkin(v) = \cosh(v) - 1$ satisfies
$\Kkin''(v) = \cosh(v) \ge 1$ on $\R$, hence is $1$-strongly
convex: for all $v_1, v_2 \in \R$ and $s \in [0,1]$, with
$v_s := (1-s)v_1 + sv_2$,
\begin{equation}\label{eq:Kkin-strong-convex}
  \Kkin(v_s)
    \;\le\; (1-s)\,\Kkin(v_1) + s\,\Kkin(v_2)
       \;-\; \tfrac{s(1-s)}{2}\,(v_1-v_2)^2.
\end{equation}
Indeed, Taylor's theorem at $v_s$ applied separately to $v_1$ and
$v_2$ gives, for some $\zeta_i$ between $v_s$ and $v_i$,
\(
  \Kkin(v_i) = \Kkin(v_s) + \Kkin'(v_s)(v_i - v_s)
    + \tfrac12\,\Kkin''(\zeta_i)(v_i - v_s)^2.
\)
Combining with weights $(1-s)$ and $s$, the linear $\Kkin'(v_s)$
contribution vanishes (because $v_s$ is the weighted mean), and the
remainders are bounded below by $\tfrac12(v_i-v_s)^2$ via
$\Kkin''\ge 1$. Using
$v_1-v_s = s(v_1-v_2)$ and $v_2-v_s = (1-s)(v_2-v_1)$, the
quadratic floor sums to $\tfrac12 s(1-s)(v_1-v_2)^2$, which
yields \eqref{eq:Kkin-strong-convex}.

Apply \eqref{eq:Kkin-strong-convex} pointwise at
$v_i = \dot\xi_i(t)$ for a.e.\ $t$. The right-hand side is
integrable: the kinetic terms by kinetic admissibility of
$\gamma_1, \gamma_2$ (Definition~\ref{def:kin-adm}), and the
quadratic term because $\dot\xi_i \in L^2([a,b])$
(Definition~\ref{def:admissible}). Lemma~\ref{lem:interp-kin}
gives $\Kkin(\dot\xi_s) \in L^1$. Integrating yields
\eqref{eq:strong-convex-A}.

For the equality clause, suppose equality holds in
\eqref{eq:strong-convex-A} for some $s_0 \in (0,1)$. Then the
pointwise inequality \eqref{eq:Kkin-strong-convex}, applied at
$v_i=\dot\xi_i(t)$, must hold with equality for a.e.\
$t\in[a,b]$: if the pointwise gap were strictly positive on a set
of positive measure, integration would give a strict integral gap,
contradicting integral equality. We now show that pointwise
equality in \eqref{eq:Kkin-strong-convex} at $s_0\in(0,1)$ forces
$v_1=v_2$.

Define, for $v,w\in\R$, the non-negative quantity
\[
  \Delta(v,w) \;:=\; \Kkin(v) - \Kkin(w) - \Kkin'(w)(v-w)
  \;=\; \int_w^v (v-u)\cosh(u)\,du,
\]
by Taylor's theorem with integral remainder for $\Kkin$ at $w$.
The identity
$(1-s_0)(v_1-v_s) + s_0(v_2-v_s) = 0$ kills the linear term in
the weighted combination, giving the exact identity
\[
  (1-s_0)\Kkin(v_1) + s_0 \Kkin(v_2) - \Kkin(v_s)
  \;=\; (1-s_0)\,\Delta(v_1,v_s) + s_0\,\Delta(v_2,v_s).
\]
Combining with $(v_1-v_s)^2=s_0^2(v_1-v_2)^2$ and
$(v_2-v_s)^2=(1-s_0)^2(v_1-v_2)^2$, the strong-convex bound
\eqref{eq:Kkin-strong-convex} rearranges to the manifestly
non-negative expression
\[
  (1-s_0)\bigl[\Delta(v_1,v_s) - \tfrac12 (v_1-v_s)^2\bigr]
  + s_0\bigl[\Delta(v_2,v_s) - \tfrac12 (v_2-v_s)^2\bigr]
  \;\ge\; 0,
\]
where each bracket is $\ge 0$ since $\cosh\ge 1$ in the integrand.
Since $s_0,(1-s_0)>0$, equality forces each bracket to vanish:
\[
  \Delta(v_i,v_s) \;=\; \tfrac12 (v_i-v_s)^2,\qquad i=1,2.
\]

We now show that $\Delta(v,w) = \tfrac12(v-w)^2$ implies $v=w$.
Subtracting,
\[
  \Delta(v,w) - \tfrac12(v-w)^2 = \int_w^v (v-u)\,[\cosh(u)-1]\,du.
\]
We claim this is \emph{strictly positive} whenever $v\ne w$,
irrespective of the ordering of $v$ and $w$. Indeed, on the open
interval with endpoints $w$ and $v$ the factor $(v-u)$ has constant
sign, and $\cosh u - 1\ge 0$ vanishes only at the single point
$u=0$; hence the product $(v-u)[\cosh u-1]$ has constant sign and is
nonzero off the null set $\{u=0\}$. If $w<v$ then $v-u>0$ on
$(w,v)$, so the integrand is $\ge0$ and positive a.e., giving
$\int_w^v(\cdots)\,du>0$. If $w>v$ then $v-u<0$ on $(v,w)$, so
$\int_v^w(\cdots)\,du<0$ and therefore
$\int_w^v(\cdots)\,du=-\int_v^w(\cdots)\,du>0$. In both cases the
quantity is strictly positive, so
$\Delta(v,w) > \tfrac12(v-w)^2$ for $v\ne w$, and the only solution
of $\Delta(v,w) = \tfrac12(v-w)^2$ is $v=w$.

Therefore $v_1 = v_s = v_2$, i.e.\ $v_1=v_2$. Applying this
pointwise, $\dot\xi_1(t)=\dot\xi_2(t)$ for a.e.\ $t\in[a,b]$. By
absolute continuity, $\xi_2-\xi_1$ is constant on $[a,b]$, i.e.\
$\gamma_2/\gamma_1$ is constant.
\end{proof}

\begin{remark}[The quantity $\Delta(v,w)$ is the cosh Bregman divergence]\label{rem:bregman-preview}
The quantity $\Delta(v,w)$ introduced in the equality clause of
the preceding proof is, up to notation, the \emph{cosh Bregman
divergence} $D_\Kkin(v\|w)$ formally defined in
Section~\ref{subsec:pythagorean} (Definition~\ref{def:bregman}).
The properties exploited above -- non-negativity, integral
representation, and the strict quadratic lower bound -- are
collected later in Proposition~\ref{prop:bregman}; we used them
inline here to keep the proof of
Theorem~\ref{thm:actionA-convex} self-contained and free of
forward references.
\end{remark}

\begin{remark}[Mechanism of the convexity theorem]\label{rem:convexity-mechanism}
The proof uses only two ingredients: (a) the log-change of
coordinates $\gamma \mapsto \xi := \log\gamma$, which converts the
multiplicative interpolation on $\Rplus$ into the affine
interpolation $\xi_s = (1-s)\xi_1 + s\xi_2$ on $\R$; and (b)
pointwise convexity of $\Kkin$, propagated by integration. The
argument is structurally a Jensen-type estimate \emph{after} the
log-change of coordinates -- no Fr\'echet differentiation, no
calculus of variations, no analytic estimate beyond pointwise
convexity. The strength comes from matching the interpolation to
the d'Alembert log-cost; Theorem~\ref{thm:pla} then converts this
strong convexity into global minimality via a single chord check.
\end{remark}

\subsection{The convex free-action principle}

The next two theorems give the convex form of the principle of least
action in the \emph{free sector}, i.e., for the kinetic action
$\actionA$ with no potential term. Their hypothesis is the weakest
possible local-minimality condition: the action does not strictly
decrease along a single positive geometric-interpolation step toward
any competitor.

\begin{theorem}[Local-min implies global-min, kinetic form]\label{thm:local-global}
Let $\gamma_{\mathrm{geo}}$ and $\gamma_{\mathrm{other}}$ be
kinetically admissible paths on $[a,b]$ sharing endpoints. If there exists
$s_0 \in (0,1]$ such that
\[
  \actionA[\gamma_{\mathrm{geo}}]
    \;\leq\; \actionA[\mathrm{interp}^\times(\gamma_{\mathrm{geo}}, \gamma_{\mathrm{other}}, s_0)],
\]
then
\begin{equation}\label{eq:local-global-quant}
  \actionA[\gamma_{\mathrm{geo}}]
    \;\leq\; \actionA[\gamma_{\mathrm{other}}]
       \;-\; \frac{1 - s_0}{2}\int_a^b \bigl(\dot\xi_{\mathrm{geo}}(t) - \dot\xi_{\mathrm{other}}(t)\bigr)^2\,dt,
\end{equation}
where $\xi_\bullet := \log\gamma_\bullet$. In particular,
$\actionA[\gamma_{\mathrm{geo}}] \le \actionA[\gamma_{\mathrm{other}}]$.
\emph{Moreover, if $s_0 \in (0,1)$, then the inequality is strict
unless $\gamma_{\mathrm{other}} \equiv \gamma_{\mathrm{geo}}$ on
$[a,b]$.}
\end{theorem}

\begin{remark}[Vacuity of the boundary case $s_0=1$]\label{rem:s0-boundary}
At $s_0 = 1$, the geometric interpolant
$\mathrm{interp}^\times(\gamma_{\mathrm{geo}},\gamma_{\mathrm{other}},1)$
collapses to $\gamma_{\mathrm{other}}$, so the hypothesis of
Theorem~\ref{thm:local-global} reduces to
$\actionA[\gamma_{\mathrm{geo}}]\le\actionA[\gamma_{\mathrm{other}}]$,
which is identical to its conclusion. The slack
$\tfrac{1-s_0}{2}\int(\dot\xi_{\mathrm{geo}} -
\dot\xi_{\mathrm{other}})^2 dt$ in \eqref{eq:local-global-quant}
also vanishes, so the strictness clause is vacuous at $s_0=1$. The
boundary value $s_0=1$ is retained in
Theorem~\ref{thm:local-global} only because the quantitative bound
\eqref{eq:local-global-quant} remains a (trivially true) identity
there; the nontrivial mathematical content lives entirely in
$s_0\in(0,1)$. For exactly this reason the global-minimality
principle Theorem~\ref{thm:pla} and its converse
(Remark~\ref{rem:pla-iff}) are stated with the \emph{interior} range
$s_0\in(0,1)$, which avoids the circularity in which the $s_0=1$
chord inequality would coincide with the conclusion itself
(Remark~\ref{rem:pla-meaning}).
\end{remark}

\begin{proof}
By Theorem~\ref{thm:actionA-convex} (the strong-convexity bound
\eqref{eq:strong-convex-A}) with $s = s_0$,
\[
  \actionA[\mathrm{interp}^\times(\gamma_{\mathrm{geo}}, \gamma_{\mathrm{other}}, s_0)]
    \;\leq\; (1-s_0)\,\actionA[\gamma_{\mathrm{geo}}]
        + s_0\,\actionA[\gamma_{\mathrm{other}}]
        - \frac{s_0(1-s_0)}{2}\!\int_a^b\!\bigl(\dot\xi_{\mathrm{geo}}-\dot\xi_{\mathrm{other}}\bigr)^2 dt.
\]
Chain with the hypothesis, subtract
$(1-s_0)\actionA[\gamma_{\mathrm{geo}}]$, and divide by $s_0 > 0$ to
obtain \eqref{eq:local-global-quant}.

For the strictness statement, assume $s_0 \in (0,1)$ and
$\gamma_{\mathrm{other}} \not\equiv \gamma_{\mathrm{geo}}$. The
shared-endpoint hypothesis forces
$\xi_{\mathrm{other}} - \xi_{\mathrm{geo}} \in H^1_0([a,b])$, so by
absolute continuity its derivative cannot be identically zero a.e.\
(otherwise the difference would be constant, hence zero by the
endpoint condition). Therefore
$\int_a^b(\dot\xi_{\mathrm{geo}} - \dot\xi_{\mathrm{other}})^2\,dt > 0$,
and the slack $\tfrac{1-s_0}{2}\int(\dot\xi_{\mathrm{geo}} -
\dot\xi_{\mathrm{other}})^2\,dt$ in \eqref{eq:local-global-quant} is
strictly positive, yielding strict inequality. At $s_0 = 1$ the
slack vanishes, so the conclusion reduces to the hypothesis.
\end{proof}

\begin{theorem}[Free-action principle, chord form]\label{thm:pla}
Let $\gamma_{\mathrm{geo}}$ be a kinetically admissible path on
$[a,b]$. Suppose that for every kinetically admissible competitor
$\gamma_{\mathrm{other}}$ sharing endpoints with $\gamma_{\mathrm{geo}}$
there exists some \emph{interior} chord parameter
$s_0 \in (0,1)$ with
\[
  \actionA[\gamma_{\mathrm{geo}}]
    \;\leq\; \actionA[\mathrm{interp}^\times(\gamma_{\mathrm{geo}}, \gamma_{\mathrm{other}}, s_0)].
\]
Then $\gamma_{\mathrm{geo}}$ is a global minimizer of $\actionA$ among
kinetically admissible competitors sharing its endpoints.
\end{theorem}

\begin{proof}
Apply Theorem~\ref{thm:local-global} for each such
$\gamma_{\mathrm{other}}$.
\end{proof}

\begin{remark}[Chord condition is a characterization of global minimality]\label{rem:pla-iff}
The chord condition of Theorem~\ref{thm:pla} is in fact equivalent to
global minimality, not merely sufficient. The non-trivial direction
(sufficiency) is Theorem~\ref{thm:pla} itself. The converse is
elementary: if $\gamma_{\mathrm{geo}}$ is a global minimizer of
$\actionA$ on the kinetically admissible class with fixed endpoints,
then for every kinetically admissible competitor
$\gamma_{\mathrm{other}}$ sharing the endpoints and every
$s \in (0,1)$, the interpolant
$\mathrm{interp}^\times(\gamma_{\mathrm{geo}}, \gamma_{\mathrm{other}}, s)$
is itself a kinetically admissible competitor sharing the endpoints
(Lemma~\ref{lem:interp-pos} together with
Lemma~\ref{lem:interp-kin}), so global minimality of
$\gamma_{\mathrm{geo}}$ gives
$\actionA[\gamma_{\mathrm{geo}}] \le \actionA[\mathrm{interp}^\times(\cdots, s)]$
trivially. Theorem~\ref{thm:pla} therefore gives a one-line
characterization of global minimality of $\actionA$ via a
directional inequality, with no calculus involved on either side.
\end{remark}

\begin{remark}[On the meaning of the free-sector claim]\label{rem:pla-meaning}
The hypothesis of Theorem~\ref{thm:pla} is a one-sided \emph{chord}
condition: for each competitor, the action does not strictly
\emph{decrease} along at least one \emph{interior} positive
geometric-interpolation step toward that competitor. This is exactly
the classical first-order optimality criterion for a convex
functional -- a point is a global minimizer iff every admissible
one-sided directional variation is non-negative -- specialized to
geometric interpolation. The mathematical content of
Theorem~\ref{thm:pla} is therefore the strong convexity of
$\actionA$ along $\mathrm{interp}^\times$
(Theorem~\ref{thm:actionA-convex}), not a new variational mechanism;
no Euler--Lagrange equation, Fr\'echet derivative, or
second-variation condition is invoked, and the genuine quantitative
output is the explicit Jensen minimizer of
Corollary~\ref{cor:uniform}.

We require the chord parameter to be \emph{interior}, $s_0\in(0,1)$,
for a substantive reason. At the boundary $s_0=1$ the interpolant
$\mathrm{interp}^\times(\gamma_{\mathrm{geo}},\gamma_{\mathrm{other}},1)$
collapses to $\gamma_{\mathrm{other}}$, so the chord inequality
$\actionA[\gamma_{\mathrm{geo}}]\le
\actionA[\mathrm{interp}^\times(\cdots,1)]$ degenerates into the
conclusion $\actionA[\gamma_{\mathrm{geo}}]\le
\actionA[\gamma_{\mathrm{other}}]$ itself and carries no deductive
content (cf.\ the vacuity of $s_0=1$ in
Remark~\ref{rem:s0-boundary}). Admitting $s_0=1$ would make the
hypothesis formally equivalent to the conclusion; restricting to
$s_0\in(0,1)$ removes this circularity while leaving the converse of
Remark~\ref{rem:pla-iff} intact. The extension to Lagrangians with
potential is necessarily weaker (Section~\ref{sec:scope}).
\end{remark}

\subsection{Closed-form minimizer and uniqueness}

\begin{corollary}[Uniform-log-velocity minimizer]\label{cor:uniform}
Given endpoints $\gamma(a) = x_a > 0$ and $\gamma(b) = x_b > 0$ on
$[a,b]$ with $a < b$, the unique minimizer of $\actionA$ among
kinetically admissible competitors
is the \emph{uniform-log-velocity path}
\[
  \gamma_*(t)
    \;=\; \exp\!\left(\log x_a + \frac{t - a}{b - a}\,(\log x_b - \log x_a)\right),
\]
whose log-velocity is the constant
$\dot\xi_* \equiv (\log x_b - \log x_a)/(b - a)$. Its kinetic action is
\[
  \actionA[\gamma_*]
    \;=\; (b - a)\left[\cosh\!\left(\tfrac{\log x_b - \log x_a}{b - a}\right) - 1\right].
\]
\end{corollary}

\begin{proof}
The path $\gamma_*$ has constant $\dot\xi_*$, hence is kinetically
admissible and has finite action. For any \emph{kinetically admissible}
$\gamma$ with these endpoints,
$\int_a^b \dot\xi\,dt = \xi(b) - \xi(a) = \log x_b - \log x_a$ is
fixed. Here $\dot\xi\in L^2([a,b])\subset L^1([a,b])$ because
$a<b<\infty$, and kinetic admissibility gives
$\Kkin(\dot\xi)\in L^1([a,b])$. Thus Jensen's inequality applies to
the probability measure $(b-a)^{-1}dt$ and the strictly convex
function $\Kkin$:
\[
  \int_a^b \Kkin(\dot\xi(t))\,dt
    \;\geq\; (b - a)\,\Kkin\!\left(\tfrac{1}{b - a}\int_a^b \dot\xi\,dt\right)
    \;=\; (b - a)\,\Kkin(\dot\xi_*),
\]
with equality iff $\dot\xi$ is a.e.\ constant, in which case
absolute continuity and the endpoint condition force
$\gamma = \gamma_*$.
\end{proof}

\subsection{Pythagorean identity and the action gap}\label{subsec:pythagorean}

Strong convexity of $\actionA$
(Theorem~\ref{thm:actionA-convex}) is reflected, at the level of the
fixed-endpoint minimizer, by an exact \emph{Bregman / Pythagorean
identity} that decomposes the action gap
$\actionA[\gamma] - \actionA[\gamma_*]$ into a non-negative
divergence integral. This refines the qualitative uniqueness of
Corollary~\ref{cor:uniform} into a quantitative
$L^2$ / Sobolev bound.

\begin{definition}[Cosh Bregman divergence]\label{def:bregman}
The \emph{cosh Bregman divergence} is the non-negative function
$D_\Kkin : \R \times \R \to \R$,
\[
  D_\Kkin(v\,\|\,w)
    \;:=\; \Kkin(v) - \Kkin(w) - \Kkin'(w)(v - w)
    \;=\; \cosh(v) - \cosh(w) - \sinh(w)(v - w).
\]
\end{definition}

\begin{proposition}[Properties of $D_\Kkin$]\label{prop:bregman}
\begin{enumerate}[label=(\roman*),noitemsep]
  \item \textbf{Integral form.}
    $D_\Kkin(v\,\|\,w) = \int_w^v (v - u)\cosh(u)\,du$.
  \item \textbf{Non-negativity.} $D_\Kkin(v\,\|\,w) \ge 0$, with
    equality iff $v = w$.
  \item \textbf{Quadratic lower bound.}
    $D_\Kkin(v\,\|\,w) \ge \tfrac{1}{2}(v - w)^2$ for all
    $v, w \in \R$.
  \item \textbf{Asymmetry near the ground state.} If $w = 0$,
    $D_\Kkin(v\,\|\,0) = \Kkin(v) = \cosh(v) - 1$, so the Bregman
    divergence \emph{at the ground velocity} reduces to the kinetic
    cost itself.
\end{enumerate}
\end{proposition}

\begin{proof}
Part (i) is Taylor's theorem with integral remainder for $\Kkin$ at
$w$: $\Kkin(v) - \Kkin(w) - \Kkin'(w)(v-w) = \int_w^v (v-u)\Kkin''(u)\,du
= \int_w^v (v-u)\cosh(u)\,du$. Parts (ii) and (iii) follow from
$\cosh \ge 1 > 0$ in (i): the integrand $(v-u)\cosh(u)$ has a
constant sign on $[\min(v,w),\max(v,w)]$, and replacing $\cosh(u)$
by its lower bound $1$ gives
$\int_w^v (v-u)\,du = \tfrac12(v-w)^2$. Part (iv) is direct evaluation.
\end{proof}

\begin{theorem}[Pythagorean identity for the geodesic minimizer]\label{thm:pythagorean}
Let $\gamma$ be a kinetically admissible path on $[a,b]$ with
$\gamma(a) = x_a$ and $\gamma(b) = x_b$, and let $\gamma_*$ be the
uniform-log-velocity minimizer of Corollary~\ref{cor:uniform}, with
constant log-velocity
$\dot\xi_* = (\log x_b - \log x_a)/(b - a)$. Writing
$\xi := \log\gamma$,
\begin{equation}\label{eq:pythagorean}
  \actionA[\gamma]
    \;=\; \actionA[\gamma_*]
    \;+\; \int_a^b D_\Kkin\!\bigl(\dot\xi(t)\,\|\,\dot\xi_*\bigr)\,dt.
\end{equation}
In particular, by Proposition~\ref{prop:bregman}(iii),
\begin{equation}\label{eq:L2-gap}
  \actionA[\gamma] - \actionA[\gamma_*]
    \;\ge\; \tfrac{1}{2}\int_a^b \bigl(\dot\xi(t) - \dot\xi_*\bigr)^2\,dt
    \;=\; \tfrac{1}{2}\,\|\dot\xi - \dot\xi_*\|_{L^2([a,b])}^2,
\end{equation}
with equality iff $\gamma \equiv \gamma_*$ on $[a,b]$.
\end{theorem}

\begin{proof}
Define $\xi_*(t) := \log x_a + \dot\xi_*\,(t - a)$, so $\xi_*$ is
affine with constant derivative $\dot\xi_*$. By
Definition~\ref{def:bregman},
\begin{equation}\label{eq:pyth-pointwise}
  \Kkin(\dot\xi(t))
    \;=\; \Kkin(\dot\xi_*) + \Kkin'(\dot\xi_*)\bigl(\dot\xi(t) - \dot\xi_*\bigr)
       \;+\; D_\Kkin\!\bigl(\dot\xi(t)\,\|\,\dot\xi_*\bigr).
\end{equation}
Before integrating, note that each of the three terms lies in
$L^1([a,b])$: the constant $\Kkin(\dot\xi_*)$ trivially; the linear
term because $\dot\xi-\dot\xi_*\in L^2([a,b])\subset L^1([a,b])$ and
$\Kkin'(\dot\xi_*)=\sinh(\dot\xi_*)$ is a finite constant; and the
Bregman term because, by \eqref{eq:pyth-pointwise}, it equals the
$L^1$ function $\Kkin(\dot\xi)$ (integrable by kinetic admissibility
of $\gamma$, Definition~\ref{def:kin-adm}) minus the two $L^1$ terms
just named. In particular $D_\Kkin(\dot\xi\,\|\,\dot\xi_*)\in
L^1([a,b])$, so the identity \eqref{eq:pythagorean} is an
equality between finite quantities.
The first term integrates to $(b-a)\Kkin(\dot\xi_*) = \actionA[\gamma_*]$
(Corollary~\ref{cor:uniform}). The second term integrates to
\(
  \Kkin'(\dot\xi_*)\int_a^b (\dot\xi - \dot\xi_*)\,dt
  = \Kkin'(\dot\xi_*)\bigl[(\xi(b)-\xi(a)) - (\xi_*(b)-\xi_*(a))\bigr]
  = 0,
\)
because $\gamma$ and $\gamma_*$ share endpoints, so
$\int_a^b\dot\xi\,dt = \log x_b - \log x_a = \int_a^b\dot\xi_*\,dt$
and the prefactor $\Kkin'(\dot\xi_*) = \sinh(\dot\xi_*)$ is a finite
constant. The third term integrates to the right-hand side of
\eqref{eq:pythagorean}. The bound \eqref{eq:L2-gap} is
Proposition~\ref{prop:bregman}(iii) integrated, with equality iff
$\dot\xi = \dot\xi_*$ a.e., equivalently $\gamma = \gamma_*$ by
absolute continuity and the endpoint condition.
\end{proof}

\begin{remark}[Pythagorean reading]\label{rem:pythagorean-reading}
Equation~\eqref{eq:pythagorean} is the Pythagorean theorem of
Bregman geometry in 1D: the geodesic $\gamma_*$ is the
\emph{Bregman projection} of any kinetically admissible competitor
$\gamma$ onto the constant-velocity submanifold subject to the
endpoint constraint, and the action gap is the Bregman ``squared
distance'' to that projection. In this sense
Corollary~\ref{cor:uniform} is not only an existence/uniqueness
statement but a projection identity, and
Theorem~\ref{thm:pla} can be read as a sub-gradient
characterization of the projection.
\end{remark}

\begin{corollary}[Sobolev / Friedrichs--Poincar\'e gap]\label{cor:sobolev-gap}
With the notation of Theorem~\ref{thm:pythagorean}, set
$\eta := \xi - \xi_* \in H^1_0([a,b])$ (so $\eta(a) = \eta(b) = 0$).
Then
\begin{equation}\label{eq:sobolev-gap}
  \actionA[\gamma] - \actionA[\gamma_*]
    \;\ge\; \tfrac{1}{2}\int_a^b \dot\eta(t)^2\,dt
    \;\ge\; \frac{\pi^2}{2(b-a)^2}\int_a^b \eta(t)^2\,dt
    \;=\; \frac{\pi^2}{2(b-a)^2}\,\bigl\|\log(\gamma/\gamma_*)\bigr\|_{L^2([a,b])}^2.
\end{equation}
The constant $\pi^2/(b-a)^2$ is the optimal Friedrichs constant on
$H^1_0([a,b])$, attained by $\eta(t) = \sin(\pi(t-a)/(b-a))$.
\end{corollary}

\begin{proof}
The first inequality is \eqref{eq:L2-gap}, written in $\eta$-coordinates:
$\dot\xi - \dot\xi_* = \dot\eta$. The second inequality is the
Friedrichs (one-dimensional Wirtinger) inequality on $H^1_0([a,b])$,
$\int \dot\eta^2 \ge \frac{\pi^2}{(b-a)^2}\int \eta^2$.
\end{proof}

\subsection{Joint convexity of the minimum-action profile}\label{subsec:perspective}

\begin{proposition}[Perspective convexity of $\actionA_*$]\label{prop:perspective}
Define the \emph{minimum-action profile}
\[
  \actionA_*(T,\Delta) \;:=\; T\bigl[\cosh(\Delta/T) - 1\bigr],
  \qquad T > 0,\;\Delta \in \R.
\]
By Corollary~\ref{cor:uniform},
$\actionA_*(b-a,\,\log x_b - \log x_a) = \actionA[\gamma_*]$ on
$[a,b]$. Then $\actionA_*$ is:
\begin{enumerate}[label=(\roman*),noitemsep]
  \item the perspective transform of $u \mapsto \cosh(u) - 1$,
    hence \emph{jointly convex} on $\Rplus \times \R$;
  \item strictly convex in $\Delta$ for each fixed $T > 0$;
  \item positively homogeneous of degree $1$:
    $\actionA_*(\lambda T,\lambda\Delta) = \lambda\,\actionA_*(T,\Delta)$
    for every $\lambda > 0$;
  \item monotone non-increasing in $T$ for each fixed $\Delta$, with
    $\partial_T \actionA_* \le 0$ and $\actionA_*(T,\Delta) \to 0$
    as $T \to \infty$.
\end{enumerate}
\end{proposition}

\begin{proof}
Let $f(u) := \cosh(u) - 1$ and write $u := \Delta/T$, so
$\actionA_*(T,\Delta) = T f(u)$. From $\partial_T u = -u/T$ and
$\partial_\Delta u = 1/T$ one computes the partial derivatives
\[
  \partial_T \actionA_* = f(u) - u f'(u),\qquad
  \partial_\Delta \actionA_* = f'(u),
\]
\[
  \partial_T^2 \actionA_* = \frac{u^2 f''(u)}{T},\qquad
  \partial_\Delta^2 \actionA_* = \frac{f''(u)}{T},\qquad
  \partial_T \partial_\Delta \actionA_* = -\frac{u f''(u)}{T}.
\]
With ordering $(T,\Delta)$, the Hessian factors as a rank-one outer product:
\[
  \nabla^2 \actionA_*(T,\Delta)
    \;=\; \frac{f''(u)}{T}\begin{pmatrix} u \\ -1 \end{pmatrix}
        \begin{pmatrix} u & -1 \end{pmatrix}
    \;=\; \frac{\cosh(\Delta/T)}{T}\begin{pmatrix} \Delta/T \\ -1 \end{pmatrix}
        \begin{pmatrix} \Delta/T & -1 \end{pmatrix}.
\]
Since $\cosh > 0$ and $T > 0$, the Hessian is positive semidefinite,
giving joint convexity (i). For fixed $T$,
$\partial_\Delta^2 \actionA_* = T^{-1}\cosh(\Delta/T) > 0$, giving
strict convexity in $\Delta$, hence (ii). Positive homogeneity
(iii) is immediate from
$\lambda T\,f(\lambda\Delta/(\lambda T)) = \lambda T f(\Delta/T)$.

For (iv), use the supporting-hyperplane inequality for the convex
function $f$ at the point $u$, evaluated at $0$:
\[
  f(0) \;\ge\; f(u) + f'(u)(0 - u) \;=\; f(u) - u f'(u),
\]
i.e.\ $f(u) - u f'(u) \le f(0) = 0$, with strict inequality when
$u \ne 0$ by strict convexity (the supporting line at $u$ touches
$f$ only at $u$). Hence
$\partial_T \actionA_* = f(\Delta/T) - (\Delta/T)f'(\Delta/T) \le 0$,
with equality iff $\Delta/T = 0$, i.e.\ iff $\Delta = 0$. The
asymptotic $\actionA_*(T,\Delta) = \Delta^2/(2T) + O(\Delta^4/T^3)
\to 0$ as $T \to \infty$ follows from the Taylor expansion
$\cosh(u) - 1 = u^2/2 + O(u^4)$ at $u = \Delta/T \to 0$.
\end{proof}

\begin{corollary}[Geodesic concatenation: subadditivity in time]\label{cor:concat}
For $T_1, T_2 > 0$ and $\Delta_1, \Delta_2 \in \R$,
\begin{equation}\label{eq:concat}
  \actionA_*(T_1 + T_2,\,\Delta_1 + \Delta_2)
    \;\le\; \actionA_*(T_1,\Delta_1) + \actionA_*(T_2,\Delta_2),
\end{equation}
with equality iff $\Delta_1/T_1 = \Delta_2/T_2$. Equivalently, the
unique global minimizer between fixed positive endpoints on
$[a,c]$ is the single uniform-log-velocity geodesic;
splitting the interval at any intermediate $b\in(a,c)$ and
relaxing the value at $b$ never beats the global geodesic, and ties
it only when the two halves continue at the same log-velocity.
\end{corollary}

\begin{proof}
By Proposition~\ref{prop:perspective}(i),(iii),
$\actionA_*$ is convex and positively homogeneous of degree $1$ on
the open half-plane $\Rplus \times \R$. A convex positively
$1$-homogeneous function is subadditive: applying convexity at
$\tfrac12((2T_1, 2\Delta_1) + (2T_2, 2\Delta_2))$ and using
$1$-homogeneity to absorb the factor $2$ yields
\eqref{eq:concat}.

For the equality clause, write $X_i=(T_i,\Delta_i)\in\Rplus\times\R$.
Equality in \eqref{eq:concat} is exactly equality in the midpoint
convexity step
$\actionA_*\!\bigl(\tfrac12(2X_1)+\tfrac12(2X_2)\bigr)
=\tfrac12\actionA_*(2X_1)+\tfrac12\actionA_*(2X_2)$, which holds iff
$\actionA_*$ is affine on the segment joining $2X_1$ and $2X_2$. By
Proposition~\ref{prop:perspective}, the Hessian
$\nabla^2\actionA_*(T,\Delta)=\tfrac{\cosh(\Delta/T)}{T}
(\Delta/T,-1)^{\!\top}(\Delta/T,-1)$ is rank one and positive
semidefinite, with its single null direction the \emph{radial} one
$(T,\Delta)$ (along which $1$-homogeneity makes $\actionA_*$ linear);
in every transverse direction it is strictly positive. Hence
$\actionA_*$ is strictly convex along any segment that does not lie
on a ray through the origin, and the midpoint identity forces
$2X_1$ and $2X_2$ -- equivalently $X_1$ and $X_2$ -- to be
positively proportional, i.e.\ $\Delta_1/T_1 = \Delta_2/T_2$.
Conversely, on a common ray $1$-homogeneity gives equality. (This
sharpening also uses the strict convexity in $\Delta$ of
Proposition~\ref{prop:perspective}(ii), which guarantees the
transverse Hessian does not degenerate.)
\end{proof}

\begin{remark}[Causal / adiabatic asymptotics]\label{rem:asymptotics}
Two limits of $\actionA_*$ are physically significant.
\begin{itemize}[noitemsep,topsep=2pt]
  \item \textbf{Short-time / ultra-relativistic.} As $T\to 0^+$
    with $\Delta\ne 0$ fixed, $\actionA_*(T,\Delta) \sim
    \tfrac12 e^{|\Delta|/T}\,T \to \infty$ exponentially. The cosh
    cost imposes a natural \emph{exponential barrier in
    log-displacement}: a log-displacement $\Delta$ over time $T$ is
    exponentially expensive when $|\Delta| \gg T$.
  \item \textbf{Long-time / adiabatic.} As $T\to\infty$ with
    $\Delta$ fixed, $\actionA_*(T,\Delta) = \tfrac{\Delta^2}{2T}
      + O(\Delta^4/T^3) \to 0$. The minimum-action profile is
    asymptotically Newtonian on the minimizer alone, even though
    no individual integrand has been linearized.
\end{itemize}
\end{remark}

\subsection{Structural reformulation: dually flat / Hessian geometry}\label{subsec:duallyflat}

\begin{remark}[Dually flat / Hessian-manifold reading]\label{rem:dually-flat}
The convexity results admit a clean reading in the language of
Hessian / dually flat manifolds (Amari--Nagaoka~\cite{amari_nagaoka2000},
Shima~\cite{shima2007}), provided
one is careful about \emph{which} Hessian structure carries the
geometric interpolation. The relevant structure lives in the
\emph{additive} coordinate $\xi=\log x$: the log-cost
$\Jlog(\xi)=\cosh\xi-1$ is a strictly convex Hessian potential on
$\R$, and the pair $(\R,\Jlog)$ is a one-dimensional dually flat
manifold with metric
\[
  g_{\Jlog}(\xi)=\Jlog''(\xi)\,d\xi^2=\cosh\xi\,d\xi^2.
\]
In this structure:
\begin{itemize}[noitemsep,topsep=2pt]
  \item geometric interpolation $\mathrm{interp}^\times$, being
    affine in $\xi$, is precisely the $e$-affine (log-affine)
    geodesic;
  \item the cosh Bregman divergence
    (Definition~\ref{def:bregman}) is the Bregman divergence
    \emph{generated by $\Jlog$}, entering the Pythagorean identity
    (Theorem~\ref{thm:pythagorean}) evaluated on the
    log-\emph{velocity} $\dot\xi$;
  \item the kinetic action $\actionA$ is the corresponding
    Bregman / cosh-energy of paths;
  \item Theorems~\ref{thm:actionA-convex},
    \ref{thm:pythagorean} and Corollary~\ref{cor:uniform} are the
    1D instances of the standard ``geodesic convexity +
    projection theorem'' for Bregman energies on dually flat
    manifolds.
\end{itemize}

\noindent\emph{This is not the Hessian metric of $\Jcost$ in the
coordinate $x$.} The pair $(\Rplus,\Jcost)$ is \emph{also} a Hessian
manifold, with affine coordinate $x$ and metric
$g_J(x)=\Jcost''(x)\,dx^2=x^{-3}\,dx^2$; its primal-affine geodesics
are parameterized linearly in $x$ (arithmetic interpolation
$\mathrm{interp}^+$), and its Levi--Civita geodesics are the
power-law family $\gamma(t)=(at+b)^{-2}$ of
Appendix~\ref{app:hessian} (Theorem~\ref{thm:geodesic-family},
Remark~\ref{rem:A-vs-E}). On the $1$-manifold $\Rplus$ these are not
distinct \emph{curves} -- a fixed-endpoint geodesic of any connection
traces the same arc of $\Rplus$ -- but distinct \emph{parameterizations}
of that common arc, each connection fixing its own affine time-law.
Because $\xi=\log x$ is a \emph{nonlinear}
reparametrization, it does not transport the affine/geodesic
structure of $x$: indeed $g_J$ pulls back to $e^{-\xi}\,d\xi^2$,
which differs from $g_{\Jlog}=\cosh\xi\,d\xi^2$. The two Hessian
structures share the carrier but use different potentials, metrics,
and connections, and only $(\R,\Jlog)$ carries the
geometric-interpolation convexity of $\actionA$. The Otto-type
geometry of $\Ehess$ on $(\Rplus,g_J)$ is left to future work.
\end{remark}

\section{Euler--Lagrange Equations and the Geodesic Picture}\label{sec:el}

Three natural Euler--Lagrange equations appear in the
cost-functional setting: one for $\actionJ$, one for $\actionA$, and
one for the Hessian path-energy $\Ehess$. All three are satisfied at the
ground state $\gamma \equiv 1$; away from it they differ in content.

\begin{definition}[Critical points used in this section]\label{def:critical}
Let $a<b$.
\begin{enumerate}[label=(\roman*),noitemsep]
  \item An admissible path $\gamma$ is a \emph{static critical point}
    of $\actionJ$ if, for every $\eta\in C^1([a,b])$ with
    $\eta(a)=\eta(b)=0$, the admissibility condition
    $\gamma+\varepsilon\eta>0$ holds for all sufficiently small
    $|\varepsilon|$, and
    \[
      \left.\frac{d}{d\varepsilon}\right|_{\varepsilon=0}
      \int_a^b \Jcost(\gamma(t)+\varepsilon\eta(t))\,dt=0.
    \]
  \item A $C^2$ log-path $\xi$ is a \emph{kinetic critical point}
    of $\actionA$ if, for every $\eta\in C^1([a,b])$ with
    $\eta(a)=\eta(b)=0$,
    \[
      \left.\frac{d}{d\varepsilon}\right|_{\varepsilon=0}
      \int_a^b \Kkin(\dot\xi(t)+\varepsilon\dot\eta(t))\,dt=0.
    \]
    For an admissible path $\gamma$ with $\xi=\log\gamma$, this
    is the critical-point notion used for $\actionA[\gamma]$.
\end{enumerate}
\end{definition}

\begin{remark}[Cost-rate Euler--Lagrange equation for $\actionJ$]\label{rem:cost-rate-EL}
Because the integrand of $\actionJ$ depends only on $\gamma$ and not
on $\dot\gamma$, the variational equation for $\actionJ$ is purely
algebraic. For admissible $\gamma$ and $\eta\in C^1([a,b])$ with
$\eta(a)=\eta(b)=0$, differentiating under the integral (justified
because $\gamma$ is continuous and strictly positive on the compact
$[a,b]$, so $\gamma+\varepsilon\eta$ stays in a compact subinterval
of $(0,\infty)$ where $\Jcost'$ is bounded, for $|\varepsilon|$ small)
gives the first variation
\[
  \delta\actionJ[\gamma](\eta)
    \;=\; \int_a^b \Jcost'(\gamma(t))\,\eta(t)\,dt.
\]
The du Bois-Reymond lemma forces $\Jcost'(\gamma(t)) = 0$ a.e.; with
$\Jcost'\circ\gamma$ continuous on $[a,b]$, this upgrades to equality
for every $t$. Since
$\Jcost'(x) = \tfrac12(1 - x^{-2})$ vanishes on $(0,\infty)$ only at
$x=1$, the unique admissible static critical point of $\actionJ$ (in
the sense of Definition~\ref{def:critical}) is the ground state
$\gamma\equiv 1$. We retain this fact only for the ground-state
coexistence observation of Remark~\ref{rem:ground-state-coexistence}
below; no theorem downstream of this remark uses the cost-rate
Euler--Lagrange equation.
\end{remark}

\subsection{The kinetic EL for \texorpdfstring{$\actionA$}{A}}

The integrand of $\actionA$ is $\Kkin(\dot\xi)$ -- a function of
$\dot\xi$ only, independent of $\xi$. The Euler--Lagrange equation is
therefore that $\partial_{\dot\xi}\Kkin$ is conserved in time:
\[
  \frac{d}{dt}\,\sinh(\dot\xi(t)) \;=\; 0
    \quad\Longleftrightarrow\quad
  \cosh(\dot\xi(t))\,\ddot\xi(t) \;=\; 0
    \quad\Longleftrightarrow\quad
  \ddot\xi(t) \;=\; 0,
\]
the last equivalence because $\cosh > 0$.

\begin{theorem}[Kinetic EL selects uniform log-velocity]\label{thm:kinetic-EL}
Let $\gamma$ be an admissible path on $[a,b]$ with
$\xi := \log\gamma \in C^2$. Then $\gamma$ is automatically kinetically
admissible (Definition~\ref{def:kin-adm}), and $\gamma$ is a kinetic
critical point of $\actionA$ in the sense of
Definition~\ref{def:critical} iff $\ddot\xi \equiv 0$, i.e., $\xi$ is
an affine function of $t$ and
\(
  \gamma(t) = x_a\,e^{\dot\xi_0(t - a)}
\)
for some $x_a > 0$ and $\dot\xi_0 \in \R$. With fixed endpoints
$\gamma(a) = x_a$, $\gamma(b) = x_b$, the unique critical point is the
uniform-log-velocity path of Corollary~\ref{cor:uniform}, and it is
simultaneously the global minimizer.
\end{theorem}

\begin{proof}
Since $\xi\in C^2([a,b])$, the derivative $\dot\xi$ is continuous and
bounded. Hence $\Kkin(\dot\xi)$ is continuous and integrable, so
$\gamma$ is kinetically admissible.

For $\eta\in C^1([a,b])$ with $\eta(a)=\eta(b)=0$, differentiating
under the integral is justified because $\dot\xi$ and $\dot\eta$ are
bounded and $\Kkin'$ is continuous on the compact range swept out by
$\dot\xi+\varepsilon\dot\eta$ for $|\varepsilon|$ small. The first
variation is
\[
  \left.\frac{d}{d\varepsilon}\right|_{\varepsilon=0}
  \int_a^b \Kkin(\dot\xi+\varepsilon\dot\eta)\,dt
  = \int_a^b \sinh(\dot\xi(t))\,\dot\eta(t)\,dt.
\]
Because $\xi\in C^2$, the function $\sinh(\dot\xi)$ is $C^1$, and
integration by parts gives
\[
  \int_a^b \sinh(\dot\xi)\,\dot\eta\,dt
  = \bigl[\sinh(\dot\xi)\eta\bigr]_a^b
    - \int_a^b \cosh(\dot\xi)\ddot\xi\,\eta\,dt
  = - \int_a^b \cosh(\dot\xi)\ddot\xi\,\eta\,dt.
\]
Thus the first variation vanishes for all endpoint-fixed
$\eta\in C^1([a,b])$ iff
$\cosh(\dot\xi)\ddot\xi=0$ on $[a,b]$ by the fundamental lemma of the
calculus of variations. Since $\cosh(\dot\xi)>0$, this is equivalent
to $\ddot\xi=0$. Therefore $\xi$ is affine, and conversely any affine
$\xi$ makes the displayed first variation vanish.

With fixed endpoints, the affine function $\xi$ is uniquely
determined, giving the path of Corollary~\ref{cor:uniform}. Its global
minimality is exactly Corollary~\ref{cor:uniform}.
\end{proof}

\begin{remark}[Regularity classes]\label{rem:EL-class}
Theorem~\ref{thm:kinetic-EL} uses the strong pointwise Euler--Lagrange
notion of ``critical point,'' which requires $\xi\in C^2$ (a proper
subclass of the kinetically admissible paths of
Definition~\ref{def:kin-adm}). The chord condition in
Theorem~\ref{thm:pla} is instead a convexity-based sufficient and
necessary condition for global minimality on the larger kinetically
admissible class, where $\xi$ need only be absolutely continuous with
$L^2$ log-derivative and finite kinetic action. For $C^2$ paths with
fixed endpoints, the uniform-log-velocity path is simultaneously the
strong-EL critical point of this theorem and the global minimizer
furnished by Corollary~\ref{cor:uniform}.
\end{remark}

\subsection{Common ground state and the Hessian-energy comparison}\label{subsec:ground-state}

So far we have compared the static $\actionJ$ and kinetic $\actionA$
Euler--Lagrange pictures. A third path-level functional also lives
on $\Rplus$ -- the Hessian Riemannian path-energy
\[
  \Ehess[\gamma] \;:=\; \int_a^b \tfrac12\,g(\gamma)\,\dot\gamma^2\,dt
             \;=\; \int_a^b \tfrac{1}{2}\,\frac{\dot\gamma(t)^2}{\gamma(t)^3}\,dt
\]
of the Hessian metric $g(x) = \Jcost''(x) = x^{-3}$, with geodesic
Euler--Lagrange equation
\begin{equation}\label{eq:geodesic}
  \ddot\gamma + \Gamma(\gamma)\,\dot\gamma^2 \;=\; 0,
  \qquad \Gamma(x) = -\frac{3}{2x}.
\end{equation}
$\Ehess$ uses the Riemannian connection of $g$, whereas $\actionA$
uses the log-affine ($e$-)connection of \S\ref{subsec:duallyflat},
and the two have different nonconstant critical families -- different
\emph{parameterizations} of the same arcs of $\Rplus$, not different
trajectories, since the carrier is $1$-dimensional
(Appendix~\ref{app:hessian}, Remark~\ref{rem:A-vs-E}). We include
$\Ehess$ here only for the comparison below; it does not enter any
subsequent theorem of this work. (The symbol $\Ehess$ is reserved
for this Hessian path-energy throughout; the unadorned $E$ later
used in \S\ref{sec:hamilton} for the conserved energy along a
trajectory is a distinct object.)

\begin{remark}[Coexistence at the ground state]\label{rem:ground-state-coexistence}
The constant path $\gamma \equiv 1$ on $[a,b]$ satisfies all three
a priori distinct critical-point conditions simultaneously: the
cost-rate Euler--Lagrange equation for $\actionJ$
(Remark~\ref{rem:cost-rate-EL}); the kinetic Euler--Lagrange
equation $\ddot\xi=0$ for $\actionA$ (Theorem~\ref{thm:kinetic-EL});
and the Hessian-metric geodesic equation \eqref{eq:geodesic}.
Verification is immediate: $\Jcost'(1) = \tfrac12(1 - 1) = 0$
gives (i); $\dot\xi \equiv 0$ gives $\ddot\xi \equiv 0$ and hence
(ii); $\dot\gamma \equiv 0$ together with $\ddot\gamma \equiv 0$
gives (iii). Although the dynamical functionals $\actionA$ and
$\Ehess$ have different nonconstant critical families -- differing in
parameterization rather than trajectory, the carrier being
$1$-dimensional -- the additional static cost-rate condition (i)
singles out $\gamma \equiv 1$ as the common normalized reference state
shared by all three critical sets. The pictures remain distinct away
from the ground state; we make no further bridging claim.
\end{remark}

\section{Conditional Bridge to Newtonian Mechanics}\label{sec:newton}

Having established in Sections~\ref{sec:cost}--\ref{sec:el} that
the d'Alembert calibration alone forces the convex free-sector
structure of $\actionA$, we now turn to the physical bridge. We
shall see that the cosh kinetic profile recovers the Newtonian
small-step limit and the relativistic rapidity dependence only
after four explicit additional postulates, which we name and
isolate at the outset:
\begin{itemize}[noitemsep,topsep=2pt]
  \item \textbf{Kinematic embedding.} The mechanical
    position $q$ is identified with the log-coordinate
    $\xi = \log\gamma$ of a positive degree of freedom.
  \item \textbf{Mass coupling.} A species carries a
    coupling constant $m > 0$ that multiplies the dimensionless
    $\Kkin$; this gives the Newtonian small-step limit and matches
    the relativistic rapidity profile of the same cosh form.
  \item \textbf{Hamiltonian primacy.} The Hamiltonian
    $H = T_H(p) + V(\xi)$ is the primary, additively combined
    cost; the Lagrangian $L = \Kkin_m - V$ is its Legendre dual,
    and the minus sign on $V$ is forced by the transform, not
    posited.
  \item \textbf{Native state cost.} The same d'Alembert
    uniqueness that forces $\Kkin$ also forces a native static
    log-cost $\Jlog(\xi) = \cosh\xi - 1$; the native Lagrangian is
    $L_{\mathrm{nat}} = \Kkin_m(\dot\xi) - k\,\Jlog(\xi)$, a
    cosh-sinh oscillator. General potentials $V$ are external
    inputs from the cost-field environment, outside the scope of
    our analysis.
\end{itemize}
These four ingredients make the Newtonian interpretation a derivation
of clearly bounded scope rather than an identification asserted at the
outset.

\subsection{Foundations of the bridge}\label{subsec:foundations}

\begin{definition}[Kinematic embedding axiom]\label{def:embedding}
\emph{Status.} The identification below is an \emph{axiom} of the
physical bridge, not a theorem of d'Alembert. We postulate that
the mechanical position coordinate $q$ of classical mechanics is
the log-coordinate $\xi=\log\gamma$ of a positive degree
of freedom, and that the dimensionless log-step
$\phi := d\xi/d\tau$ is calibrated as a boost rapidity. Neither
identification is forced by the convexity theorems of
Sections~\ref{sec:cost}--\ref{sec:el}; both are physical
postulates whose role is exactly to supply the kinematic content
that d'Alembert's algebra alone cannot.

\smallskip
\noindent\emph{Statement.}

Fix the species-relevant invariant speed $c > 0$ and a reference
time scale $t_0>0$. We write
\[
  \tau := t/t_0
\]
for the associated dimensionless evolution parameter. Throughout
Sections~\ref{sec:newton}--\ref{sec:hamilton} we work in
\emph{natural units}, in which $c=1$ and $t_0=1$, and restore
standard units explicitly at each appearance of a dimensional
quantity.

For the physical bridge in
Sections~\ref{sec:newton}--\ref{sec:hamilton} we identify the
variational parameter used in the preceding sections with $\tau$.
Thus a dot in these sections denotes differentiation with respect
to the dimensionless parameter $\tau$ unless a physical-time
derivative such as $dq/dt$ is written explicitly. In natural units
$t=\tau$, so the dot notation agrees with the usual mechanical
notation.

We identify a positive degree of freedom
$\gamma \in \Rplus$ with classical kinematic data via
\[
  \xi \;:=\; \log\gamma \;\in\;\R, \qquad
  q \;:=\; \xi \quad\text{(natural units)}
  \quad\Bigl[\text{equivalently } q := c\,t_0\,\xi\text{ in standard units}\Bigr],
\]
and we identify the dimensionless log-step per unit dimensionless
time with the boost rapidity:
\begin{equation}\label{eq:rapidity-velocity}
  \phi \;:=\; \frac{d\xi}{d\tau},\qquad
  \vsr/c \;:=\; \tanh(\phi),\qquad
  \frac{dq}{dt} \;=\; c\,\frac{d\xi}{d\tau}
  \quad\Bigl[= \dot\xi = \phi \text{ in natural units } c=t_0=1\Bigr].
\end{equation}
The mechanical ground state $q = 0$ corresponds to the
ground state $\gamma = 1$.
\end{definition}

\begin{remark}[Dimensional dictionary]\label{rem:dim-dict}
Sections~\ref{sec:newton}--\ref{sec:hamilton} are written in natural
units ($c = t_0 = 1$), so every symbol is either dimensionless or
carries some power of mass. The following table lists each symbol's
role together with its standard-unit (i.e., $c, t_0$ restored)
dimensions, so that the natural-unit expressions in the body can be
read in either convention.

\smallskip
\noindent\begin{tabular}{@{}lll@{}}
\hline
Symbol & Description & Standard-unit dimensions \\ \hline
$\xi$ & log-coordinate, $\xi=\log\gamma$ & dimensionless \\
$\tau = t/t_0$ & dimensionless time & dimensionless \\
$t$ & physical time & $T$ \\
$t_0$ & reference time scale & $T$ \\
$c$ & invariant speed & $LT^{-1}$ \\
$\phi = d\xi/d\tau$ & rapidity (boost parameter) & dimensionless \\
$\vsr = c\tanh\phi$ & SR $3$-velocity & $LT^{-1}$ \\
$q = c\,t_0\,\xi$ & mechanical position & $L$ \\
$dq/dt = c\,\phi$ & coordinate velocity & $LT^{-1}$ \\
$m$ & mass coupling (Def.~\ref{def:mass}) & $M$ \\
$k$ & binding coupling (\S\ref{subsec:oscillator}) & $M L^2 T^{-2}$ (energy) \\
$\Kkin_m(\phi) = m c^2 (\cosh\phi - 1)$ & physical kinetic cost & $M L^2 T^{-2}$ (energy) \\
$p = \partial L/\partial\dot\xi$ & conjugate to $\xi$ & $M L^2 T^{-2}$ (energy) \\
$L,\,H$ & Lagrangian, Hamiltonian & $M L^2 T^{-2}$ (energy) \\
$\actionA[\gamma]$ & free-sector kinetic action & dimensionless \\
$m c^2 t_0\,\actionA[\gamma],\ \actionL_V[\gamma]=\int L\,dt$
  & physical action & $M L^2 T^{-1}$ (action) \\
\hline
\end{tabular}

\smallskip
\noindent\emph{Notes.} (i) Because $\xi$ is dimensionless, the
conjugate $p := \partial L/\partial\dot\xi$ has dimensions of
\emph{energy}, not of mechanical momentum $M L T^{-1}$; the
conventional SR-mechanical momentum is $p/c$, with units
$M L T^{-1}$. (ii) The position identification $q = c\,t_0\,\xi$ in
standard units is the dimensionally consistent reading that makes
$dq/dt = c\,\phi$ a proper velocity ($LT^{-1}$); in natural units it
collapses to $q = \xi$ and $\dot q = \phi$ as written. (iii) In
natural units the physical kinetic cost $\Kkin_m(\phi)$ and the
binding cost $k(\cosh\xi-1)$ both reduce to mass times a
dimensionless cosh-bump; the explicit $c^2$ factors are restored only
in expressions written outside of natural units. (iv) The
free-sector functionals $\actionA[\gamma]$ and
$\actionA_*(T,\Delta)$ of
Sections~\ref{sec:pathspace}--\ref{sec:convexity} are
\emph{dimensionless}, consistent with Remark~\ref{rem:notation}; the
symbol $\actionA$ acquires action dimensions only through the
mass--time prefactor, as the physical kinetic action
$m c^2 t_0\,\actionA[\gamma]$, which in the free case $V\equiv 0$
coincides with $\actionL_V[\gamma]=\int L\,dt$. The two rows above
list these two distinct objects separately.
\end{remark}

\begin{remark}[Two velocity notions and why rapidity is primary]\label{rem:two-velocities}
Two observations force the rapidity identification of
Definition~\ref{def:embedding} rather than a direct
``$\dot\xi = $ Newtonian (SR $3$-)velocity'' reading, and
Equation~\eqref{eq:rapidity-velocity} carries two distinct
``velocity-like'' quantities that must not be confused.

\smallskip
\noindent\textbf{(a) Algebra.} d'Alembert's equation
\eqref{eq:dAlembert} on the additive variable $\xi$ is precisely
the \emph{composition law for boosts}: $\cosh(\phi_1 + \phi_2)
+ \cosh(\phi_1 - \phi_2) = 2\cosh(\phi_1)\cosh(\phi_2)$. The
canonical additive parameter on which this composition is linear
is the rapidity $\phi$, not the SR $3$-velocity $\vsr$ (which
composes nonlinearly via the Einstein addition formula).

\smallskip
\noindent\textbf{(b) Dimensions.} If $\gamma$ is dimensionless,
then so is $\xi$. Differentiating $\xi$ with respect to physical
time gives a quantity with dimensions of inverse time, so we
compare instead with the dimensionless step
$\phi = d\xi/d\tau$, where $\tau=t/t_0$. The Newtonian
identification ``$q=\xi$'' is dimensionally consistent only after
choosing natural units; in standard units it becomes $q=c\,t_0\,\xi$
and $dq/dt = c\,d\xi/d\tau = c\,\phi$ (Remark~\ref{rem:dim-dict}).
The rapidity identification \eqref{eq:rapidity-velocity} is
dimension-free because it uses $\phi$, not the unscaled derivative
with respect to physical time.

\smallskip
\noindent\textbf{(c) Two velocity notions: $\vsr$ versus $\dot q$.}
Under the kinematic axiom \eqref{eq:rapidity-velocity}, the
\emph{coordinate (log-) velocity}
$\dot q = dq/dt = c\,\phi
   \overset{\text{nat.~units}}{=} \phi$
is the rate of change of the log-coordinate $q=\xi$ with respect
to lab time; under the kinematic axiom this \emph{is the rapidity
itself} (in natural units). The
\emph{special-relativistic $3$-velocity} $\vsr := c\tanh(\phi)$
is defined by the rapidity-to-velocity formula of special
relativity. The two agree only to leading order. In natural units,
where $\dot q=\phi$ and $c=1$,
\[
  \vsr = \tanh(\dot q)
       = \dot q - \frac{\dot q^3}{3} + O(\dot q^5),
\]
while in standard units, using $\phi=\dot q/c$,
\[
  \vsr = c\tanh(\dot q/c)
       = \dot q - \frac{\dot q^3}{3c^2} + O(\dot q^5/c^4).
\]
For
$|\dot q|\ll 1$ they coincide; for large $|\dot q|$, $\vsr$
saturates at $\pm c$ while $\dot q=\phi$ is unbounded. The
Newtonian small-step limit (Proposition~\ref{prop:rapidity}(ii)
and Theorem~\ref{thm:small-v}) lives where the distinction is
invisible at leading order; the relativistic profile match
(Proposition~\ref{prop:rapidity}(i)) lives where the distinction
is sharp. We always use $\vsr$ for the SR $3$-velocity
$c\tanh\phi$, distinct from $\dot q$ except at $\phi=0$.

\smallskip
\noindent The Newtonian-limit identification of the log-step
variable with velocity in natural units is therefore not an
independent embedding; it is the small-rapidity consequence
$\phi \approx \vsr/c$ for $|\vsr|\ll c$
(Proposition~\ref{prop:rapidity}(ii)). Under the kinematic axiom,
$\dot q = \phi$ \emph{exactly} in natural units.
\end{remark}

With Definition~\ref{def:embedding} in place, the results of
Sections~\ref{sec:convexity}--\ref{sec:el} are re-readable as
statements about a genuine scalar mechanical coordinate $q$.

\begin{definition}[Species mass coupling]\label{def:mass}
Each particle species carries a \emph{mass coupling} $m > 0$ (with
dimensions of mass), a species-specific prefactor on the
dimensionless d'Alembert-forced kinetic cost. The \emph{physical
kinetic cost} is, in natural units ($c=t_0=1$),
\[
  \Kkin_m(\phi) \;:=\; m\,\Kkin(\phi)
    \;=\; m\bigl[\cosh(\phi) - 1\bigr],
\]
and in standard units the same object reads
$\Kkin_m(\phi) = m c^2[\cosh(\phi) - 1]$, with dimensions of
energy. When the dimensionless parameter is chosen as $\tau=t/t_0$,
the physical action includes the corresponding time factor
$dt=t_0\,d\tau$. We write $\Kkin_m$ throughout for the natural-unit
form, where $c=t_0=1$ and $\phi=\dot\xi$.
\end{definition}

\begin{proposition}[Single cosh profile covers Newtonian and rapidity regimes]\label{prop:rapidity}
Let $m>0$. With the rapidity identification
\eqref{eq:rapidity-velocity} of Definition~\ref{def:embedding}, the
notational distinction $\dot q = \phi$ versus
$\vsr = c\tanh\phi$ of Remark~\ref{rem:two-velocities} in force:
\begin{enumerate}[label=(\roman*),noitemsep]
  \item \textbf{Relativistic rapidity profile (as a function of
    rapidity, or equivalently of $\vsr$).} For all
    rapidities $\phi \in \R$,
    \[
      \Kkin_m(\phi) \;=\; m\bigl[\cosh(\phi) - 1\bigr]
                  \;=\; m(\gamma_L - 1),
    \]
    where $\gamma_L := \cosh\phi = 1/\sqrt{1 - (\vsr/c)^2}$ is the
    Lorentz factor associated with $\vsr/c = \tanh\phi$. This is the
    exact relativistic kinetic-energy dependence on rapidity; in
    standard units the right-hand side reads $m c^2(\gamma_L - 1)$.
    This is a profile match for the d'Alembert step cost, not a
    claim that the Legendre Hamiltonian below is the standard
    special-relativistic Hamiltonian.
  \item \textbf{Newtonian limit (consequence).} For $|\phi| \ll 1$
    (equivalently $|\vsr| \ll c$, equivalently $|\dot q|\ll 1$ in
    natural units), the Taylor expansion of $\cosh$ gives
    \[
      \Kkin_m(\phi)
        \;=\; \tfrac12 m\phi^2 + \tfrac{m}{24}\phi^4 + O(\phi^6).
    \]
    Substituting either of the two leading-order equivalences
    $\phi = \dot q$ (exact, by~\eqref{eq:rapidity-velocity}) or
    $\phi = \vsr/c + (\vsr/c)^3/3 + O((\vsr/c)^5)$ gives, in
    natural units $c=t_0=1$,
    \[
      \Kkin_m \;=\; \tfrac12 m\,\dot q^2 + \tfrac{m}{24}\dot q^4 + O(\dot q^6)
              \;=\; \tfrac12 m\,\vsr^2 + O(\vsr^4),
    \]
    the standard Newtonian kinetic energy. The two readings
    ``$\dot q$'' and ``$\vsr$'' coincide to leading order $v^2$ and
    differ only at $O(v^4)$.
\end{enumerate}
\end{proposition}

\begin{proof}
(i) is the identity
$\cosh(\operatorname{artanh}\beta) = (1-\beta^2)^{-1/2}$ applied to
$\beta=\vsr/c$ and $\phi=\operatorname{artanh}(\vsr/c)$, multiplied
by $m$; equivalently, $\cosh\phi = \gamma_L$ by definition of
rapidity.

(ii) The first expansion is the standard cosh Taylor series
$\cosh(\phi) - 1 = \tfrac12\phi^2 + \tfrac{1}{24}\phi^4 + O(\phi^6)$
multiplied by $m$, with the sharp explicit error bound
$|\Kkin_m(\phi) - \tfrac12 m\phi^2| \le \tfrac{m\phi^4}{24(1-\phi^2)}$
for $|\phi|<1$ (Proposition~\ref{prop:Kkin}(iv) multiplied by $m$).

Substituting $\phi=\dot q$ (an \emph{exact} identification under the
kinematic axiom~\eqref{eq:rapidity-velocity} in natural units)
directly yields $\Kkin_m = \tfrac12 m\dot q^2 + \tfrac{m}{24}\dot
q^4 + O(\dot q^6)$.

Substituting instead the SR $3$-velocity expansion
$\phi = \operatorname{artanh}(\vsr/c) =
\vsr/c + (\vsr/c)^3/3 + O((\vsr/c)^5)$ gives, in natural units,
$\phi^2 = \vsr^2 + 2\vsr\cdot \vsr^3/3 + O(\vsr^6) = \vsr^2 +
O(\vsr^4)$, so $\tfrac12 m\phi^2 = \tfrac12 m\vsr^2 + O(\vsr^4)$,
and the quartic-and-higher tail $\tfrac{m}{24}\phi^4 +O(\phi^6) = O(\vsr^4)$
is absorbed in the same remainder.
The agreement of the two readings at the $\vsr^2$ level is then
$\dot q^2 = \vsr^2 + O(\vsr^4)$, which follows from $\dot q = \phi
= \vsr/c + O((\vsr/c)^3)$.
\end{proof}

\begin{remark}[Provenance of the mass value]
The mass coupling $m$ is the only particle-species parameter in
the kinetic sector. We do not produce its provenance -- why a
given species carries a given numerical $m$ -- here; it is
determined by deeper structural properties of the species that lie
outside the scope of this paper. Within our scope, $m$ is an input.
What we derive is the \emph{form} $m\Kkin(\phi)$ of the kinetic
Lagrangian, with the same $\cosh - 1$ function covering the
Newtonian small-step limit and the rapidity profile.
\end{remark}

\begin{definition}[Additive cost postulate]\label{def:additive}
The \emph{total instantaneous cost} of a trajectory in a
degree of freedom $\xi$ with conjugate momentum $p$ decomposes
additively into a velocity cost and a state cost:
\[
  H(\xi, p) \;=\; T_H(p) \;+\; V(\xi),
\]
where $T_H$ is the Legendre transform of the physical kinetic cost
$\Kkin_m$ and $V : \R \to \R$ is the state cost. $H$ is the
\emph{Hamiltonian} and is the primary dynamical object of the
framework.
\end{definition}

\begin{proposition}[The cosh kinetic Hamiltonian]\label{prop:TH}
Let $m>0$. For the physical kinetic cost
$\Kkin_m(\dot\xi) = m[\cosh\dot\xi - 1]$,
\[
  p \;=\; \frac{\partial\Kkin_m}{\partial\dot\xi} \;=\; m\sinh\dot\xi,
  \qquad
  \dot\xi \;=\; \operatorname{arsinh}\!\left(\frac{p}{m}\right),
\]
and the Legendre transform is
\[
  T_H(p) \;=\; p\dot\xi - \Kkin_m(\dot\xi)
         \;=\; p\,\operatorname{arsinh}\!\left(\frac{p}{m}\right)
             - \sqrt{m^2 + p^2} + m.
\]
Small-$p$ Taylor:
\[
  T_H(p) \;=\; \frac{p^2}{2m} \;-\; \frac{p^4}{24\,m^3} \;+\; O(p^6/m^5).
\]
\end{proposition}

\begin{proof}
The identities $p = m\sinh\dot\xi$ and $\dot\xi = \operatorname{arsinh}(p/m)$
are immediate from the definition of $p$ and the
inverse-hyperbolic-sine. Using
$\cosh(\operatorname{arsinh}(u)) = \sqrt{1 + u^2}$ with $u = p/m$,
\(
  \Kkin_m(\dot\xi) = m[\sqrt{1 + p^2/m^2} - 1] = \sqrt{m^2 + p^2} - m,
\)
and therefore
\(
  T_H(p) = p\,\operatorname{arsinh}(p/m) - \sqrt{m^2 + p^2} + m.
\)
For the Taylor expansion:
$\operatorname{arsinh}(p/m) = p/m - p^3/(6 m^3) + O(p^5/m^5)$, so
$p\operatorname{arsinh}(p/m) = p^2/m - p^4/(6 m^3) + O(p^6/m^5)$;
and
$\sqrt{m^2 + p^2} = m + p^2/(2m) - p^4/(8 m^3) + O(p^6/m^5)$.
Subtracting,
\[
  T_H(p)
    = \tfrac{p^2}{m} - \tfrac{p^4}{6 m^3} - m - \tfrac{p^2}{2m}
      + \tfrac{p^4}{8 m^3} + m + O(p^6/m^5)
    = \tfrac{p^2}{2m}
      + \left(\tfrac{1}{8} - \tfrac{1}{6}\right)\tfrac{p^4}{m^3} + O(p^6/m^5)
    = \tfrac{p^2}{2m} - \tfrac{p^4}{24\,m^3} + O(p^6/m^5).
\]
\end{proof}

\begin{proposition}[Cosh-dual Hamiltonian is distinct from the SR free-particle Hamiltonian]\label{prop:not-SR}
Let $m > 0$. The cosh-dual Hamiltonian
\[
  T_H(p) \;=\; p\,\operatorname{arsinh}\!\Bigl(\frac{p}{m}\Bigr)
              \;-\; \sqrt{m^2 + p^2} \;+\; m
\]
of Proposition~\ref{prop:TH} and the special-relativistic
free-particle Hamiltonian (with rest energy subtracted)
\[
  H_{\mathrm{SR}}(p) \;:=\; \sqrt{m^2 + p^2} \;-\; m
\]
are distinct functions on $\R$. More precisely:

\begin{enumerate}[label=\textup{(\roman*)},leftmargin=*]
  \item \textbf{Strict separation away from zero.} Let
  $f(p) := T_H(p) - H_{\mathrm{SR}}(p)
        = p\,\operatorname{arsinh}(p/m) - 2\sqrt{m^2+p^2} + 2m$.
  Then
  \[
    f''(p) \;=\; \frac{p^2}{(m^2 + p^2)^{3/2}} \;\ge\; 0,
  \]
  with equality only at $p = 0$, so $f$ is convex on $\R$ and
  strictly convex off the single point $p = 0$. Together with
  $f(0) = 0$ and $f'(0) = 0$, this gives
  \[
    T_H(p) \;>\; H_{\mathrm{SR}}(p) \qquad\text{for every } p \neq 0.
  \]

  \item \textbf{Newtonian agreement, quartic divergence.} The two
  Hamiltonians share the Newtonian small-$p$ limit at order $p^2$
  but differ at order $p^4$:
  \begin{align*}
    T_H(p) &\;=\; \frac{p^2}{2m} - \frac{p^4}{24\,m^3} + O(p^6/m^5),\\
    H_{\mathrm{SR}}(p) &\;=\; \frac{p^2}{2m} - \frac{p^4}{8\,m^3} + O(p^6/m^5),
  \end{align*}
  so
  \[
    T_H(p) - H_{\mathrm{SR}}(p)
      \;=\; \frac{p^4}{12\,m^3} + O(p^6/m^5).
  \]
\end{enumerate}
The match of Proposition~\ref{prop:rapidity} is therefore one of
\emph{profile} in rapidity -- kinetic energy as a function of
rapidity equals $m(\cosh\phi - 1) = m(\gamma_L - 1)$ in both
theories -- not an identity of Hamiltonians. We use the cosh-dual
$T_H$ throughout.
\end{proposition}

\begin{proof}
\textit{(i)} Set $f(p) := T_H(p) - H_{\mathrm{SR}}(p)
= p\,\operatorname{arsinh}(p/m) - 2\sqrt{m^2+p^2} + 2m$. Using
$(d/dp)\operatorname{arsinh}(p/m) = 1/\sqrt{m^2+p^2}$, direct
differentiation gives
\[
  f'(p) \;=\; \operatorname{arsinh}\!\Bigl(\frac{p}{m}\Bigr)
            \;+\; \frac{p}{\sqrt{m^2+p^2}}
            \;-\; \frac{2p}{\sqrt{m^2+p^2}}
        \;=\; \operatorname{arsinh}\!\Bigl(\frac{p}{m}\Bigr)
            \;-\; \frac{p}{\sqrt{m^2+p^2}},
\]
and differentiating once more, using
$(d/dp)\bigl[p/\sqrt{m^2+p^2}\bigr] = m^2/(m^2+p^2)^{3/2}$,
\[
  f''(p) \;=\; \frac{1}{\sqrt{m^2+p^2}}
              \;-\; \frac{m^2}{(m^2 + p^2)^{3/2}}
         \;=\; \frac{(m^2 + p^2) - m^2}{(m^2 + p^2)^{3/2}}
         \;=\; \frac{p^2}{(m^2 + p^2)^{3/2}}.
\]
Hence $f''(p) \ge 0$ on $\R$, with equality iff $p = 0$. In
particular $f$ is convex on $\R$ and strictly convex on any open
interval not containing $0$. The critical-point equations
$f(0) = 0\cdot 0 - 2m + 2m = 0$ and $f'(0) = 0 - 0 = 0$ identify
$p = 0$ as the unique critical point and, by convexity, the global
minimum of $f$. Suppose, toward a contradiction, that
$f(p_0) = 0$ for some $p_0 \neq 0$. By convexity of $f$ on
$[0, p_0]$ (or $[p_0,0]$ if $p_0 < 0$) with $f(0) = f(p_0) = 0$,
$f \le 0$ on this interval; combined with $f \ge 0$ from
minimality at $0$, this forces $f \equiv 0$ on the interval, hence
$f'' \equiv 0$ there. But $f''(p) > 0$ at every interior point
$p \ne 0$, contradicting $f''$ vanishing on an open subinterval.
Therefore $f(p) > 0$ for every $p \ne 0$, i.e.,
$T_H(p) > H_{\mathrm{SR}}(p)$ for every $p \ne 0$.

\textit{(ii)} The $T_H$ expansion is Proposition~\ref{prop:TH}. For
$H_{\mathrm{SR}}$, apply the binomial series
$\sqrt{1+u} = 1 + u/2 - u^2/8 + O(u^3)$ at
$u = p^2/m^2 \ge 0$ to obtain
\[
  \sqrt{m^2 + p^2} \;=\; m\,\sqrt{1 + p^2/m^2}
    \;=\; m \;+\; \frac{p^2}{2m} \;-\; \frac{p^4}{8\,m^3}
       \;+\; O(p^6/m^5),
\]
and subtract $m$ to get the displayed $H_{\mathrm{SR}}$
expansion. Subtracting the two expansions term by term, the $p^2$
terms cancel and the $p^4$ coefficient of $T_H - H_{\mathrm{SR}}$
is
\[
  -\,\frac{1}{24\,m^3} \;-\; \Bigl(-\frac{1}{8\,m^3}\Bigr)
    \;=\; \frac{-1 + 3}{24\,m^3} \;=\; \frac{1}{12\,m^3},
\]
giving $T_H(p) - H_{\mathrm{SR}}(p) = p^4/(12\,m^3) + O(p^6/m^5)$
as claimed. The positivity of this leading $p^4$ coefficient is
also implied by part (i): $f(p) > 0$ for small $p \ne 0$ together
with $f(p) = c\,p^4 + O(p^6)$ from the Taylor expansion forces
$c \ge 0$, and the explicit value $c = 1/(12\,m^3)$ confirms
$c > 0$.
\end{proof}

\begin{remark}[The Legendre sign is derived, not posited]\label{rem:sign}
From the additive postulate $H = T_H(p) + V(\xi)$, the Lagrangian
is derived via the inverse Legendre transform:
\[
  L(\xi, \dot\xi) \;=\; p\dot\xi - H
    \;=\; \bigl[p\dot\xi - T_H(p)\bigr] - V(\xi)
    \;=\; \Kkin_m(\dot\xi) - V(\xi).
\]
The minus sign on $V$ in the Lagrangian is \emph{forced} by the
Legendre transform of an additive Hamiltonian; it is not an
independent postulate. In a Hamiltonian-primary formulation,
where the primary object is the additive
cost $H$, the sign structure $L = T - V$ is a theorem, not an
axiom.
\end{remark}

In the free-sector analysis, the d'Alembert log-cost
$\Jlog(\xi) = \cosh\xi - 1$ of \S\ref{sec:cost} played two distinct
roles. Evaluated at the log-velocity $\dot\xi$, it produced the
kinetic integrand $\Kkin(v) = \Jlog(v) = \cosh v - 1$ that drives
the convexity theorem of \S\ref{sec:convexity}. Evaluated at the
log-position $\xi$ itself, the same log-cost is the integrand of
the velocity-free static cost
$\actionJ[\gamma]=\int\Jlog(\xi)\,dt$ of \S\ref{sec:pathspace},
whose algebraic Euler--Lagrange equation
(Remark~\ref{rem:cost-rate-EL}) selects the ground state
$\gamma\equiv 1$. The d'Alembert calibration therefore fixes
$\Jlog$ on \emph{both} arguments. Assigning the static cost its own
species prefactor $k > 0$ -- the \emph{binding coupling} -- gives
the \emph{native d'Alembert state cost}
$V_{\mathrm{nat}}(\xi) = k\Jlog(\xi) = k(\cosh\xi - 1)$. Thus
$\Jlog$ enters the native Lagrangian below twice under the same
calibration: once at log-velocity (kinetic) and once at log-position
(potential), with independent couplings $m$ and $k$.

\begin{definition}[Native d'Alembert Lagrangian]\label{def:native}
With kinetic mass coupling $m$ and binding coupling $k$, the
\emph{native d'Alembert Lagrangian} is
\[
  L_{\mathrm{nat}}(\xi, \dot\xi)
    \;:=\; \Kkin_m(\dot\xi) - k\,\Jlog(\xi)
    \;=\; m\bigl[\cosh(\dot\xi) - 1\bigr] - k\bigl[\cosh(\xi) - 1\bigr].
\]
In the \emph{pure d'Alembert case} $k = m$, the constants cancel
and
\(
  L_{\mathrm{nat}}(\xi, \dot\xi) = m\bigl[\cosh(\dot\xi) - \cosh(\xi)\bigr].
\)
\end{definition}

\begin{remark}[Provenance of the binding coupling]\label{rem:k-provenance}
Like the mass coupling $m$ of Definition~\ref{def:mass}, the
binding coupling $k > 0$ is treated here as an external species
parameter, subject to the same provenance caveat recorded for $m$ in
the remark above: its numerical value is governed by the same
upstream structural mechanisms that fix $m$ and lies outside the
scope of this paper. What we derive is the \emph{form}
of the native Lagrangian, with the d'Alembert log-cost $\Jlog$
appearing under the same calibration in both the kinetic slot
(evaluated at $\dot\xi$) and the potential slot (evaluated at
$\xi$); the two couplings $(m,k)$ then enter as independent
prefactors of those two slots.
\end{remark}

\begin{remark}[Native vs.\ general potentials]
$L_{\mathrm{nat}}$ is the minimal dynamical object produced by
d'Alembert alone: kinetic cost plus native static cost, combined
via the Hamiltonian-primary Legendre transform
(Remark~\ref{rem:sign}). A \emph{general} potential $V(\xi)$ in
Section~\ref{subsec:general} is the replacement of the native
$k\Jlog(\xi)$ by whatever interaction model supplies the state cost.
Here, any $V \neq k\Jlog$ is an external input that we do not
derive from d'Alembert.
\end{remark}

\subsection{Small-velocity reduction of the kinetic action}\label{subsec:small-v}

\begin{theorem}[Small-velocity reduction, mass-coupled]\label{thm:small-v}
Let $\gamma$ be an admissible path on $[a,b]$ with $\xi := \log\gamma$
and $\|\dot\xi\|_\infty \leq 1/10$ (hence in particular kinetically
admissible, since $\cosh(\dot\xi)\le\cosh(1/10)<\infty$ a.e.), and let
$m > 0$. Then
\begin{equation}\label{eq:small-v-reduction}
  \left|\,m\,\actionA[\gamma]
    \;-\; \tfrac{m}{2}\int_a^b \dot\xi(t)^2\,dt\,\right|
    \;\leq\; \frac{m}{24}\cdot\frac{100}{99}\int_a^b \dot\xi(t)^4\,dt
    \;\leq\; \frac{m}{2376}\int_a^b \dot\xi(t)^2\,dt.
\end{equation}
Under the kinematic embedding $q := \xi$ and the natural-unit
calibration of the variational parameter with $\tau$, the
mass-weighted kinetic action $m\,\actionA[\gamma]$ reduces to the
standard kinetic integral
$T[q] = \tfrac12\int_a^b m\,\dot q(t)^2\,dt$ with relative error at
most $1/1188 \approx 8.4 \times 10^{-4}$.
\end{theorem}

\begin{proof}
The sharp quartic remainder \eqref{eq:Kkin-sharp-quartic} applied
pointwise at $v=\dot\xi(t)$ with $|\dot\xi(t)|\le 1/10$ gives
$|\Kkin(\dot\xi)-\tfrac12\dot\xi^2|\le \dot\xi^4/24\cdot 100/99$,
which integrated against $m\,dt$ yields the first inequality of
\eqref{eq:small-v-reduction}. The second uses
$\dot\xi^4\le \dot\xi^2/100$ when $|\dot\xi|\le 1/10$, giving a
relative-error bound of $2/2376 = 1/1188$ on the kinetic
approximation.
\end{proof}

\begin{remark}
The kinetic term emerges from the Taylor of $\Kkin(\dot\xi)$ --
\emph{velocity}, not of $\Jcost(1 + \varepsilon)$ -- state
displacement. The mass coupling $m$ enters as a uniform prefactor
consistent with the Newtonian small-step limit and the rapidity
profile of Proposition~\ref{prop:rapidity}.
\end{remark}

\subsection{The native cosh-sinh oscillator}\label{subsec:oscillator}

\begin{theorem}[EL equation and small-amplitude limit of $L_{\mathrm{nat}}$]\label{thm:native-EL}
For the native d'Alembert Lagrangian
(Definition~\ref{def:native}) with couplings $m, k > 0$, the
Euler--Lagrange equation is
\begin{equation}\label{eq:cosh-sinh}
  m\cosh(\dot\xi)\,\ddot\xi + k\sinh(\xi) \;=\; 0.
\end{equation}
In the small-velocity, small-amplitude limit
($|\dot\xi|, |\xi| \ll 1$), \eqref{eq:cosh-sinh} reduces to the
linear harmonic oscillator
\(
  m\,\ddot\xi + k\,\xi = 0,
\)
with angular frequency $\omega = \sqrt{k/m}$.
\end{theorem}

\begin{proof}
$\partial_{\dot\xi}L_{\mathrm{nat}} = m\sinh\dot\xi$, so
$(d/dt)\partial_{\dot\xi}L_{\mathrm{nat}} = m\cosh\dot\xi\,\ddot\xi$.
$\partial_\xi L_{\mathrm{nat}} = -k\sinh\xi$. The Euler--Lagrange
equation $(d/dt)\partial_{\dot\xi}L - \partial_\xi L = 0$ gives
\eqref{eq:cosh-sinh}. Taylor-expanding $\cosh\dot\xi = 1 + O(\dot\xi^2)$
and $\sinh\xi = \xi + O(\xi^3)$ recovers the harmonic oscillator.
\end{proof}

\begin{corollary}[Conserved energy of the native oscillator]\label{cor:native-energy}
For couplings $m,k>0$, the native Hamiltonian
(Proposition~\ref{prop:TH} with $V = k\Jlog$),
\[
  H_{\mathrm{nat}}(\xi, p)
    \;=\; p\,\operatorname{arsinh}\!\left(\tfrac{p}{m}\right)
        - \sqrt{m^2 + p^2} + m + k\bigl[\cosh(\xi) - 1\bigr],
\]
is constant along any $C^2$ solution of \eqref{eq:cosh-sinh}. Its
small-$(\xi, p)$ expansion is
\(
  H_{\mathrm{nat}} = \tfrac{p^2}{2m} + \tfrac{k}{2}\xi^2
  - \tfrac{p^4}{24\,m^3} + \tfrac{k}{24}\xi^4 + O(p^6, \xi^6),
\)
the harmonic oscillator Hamiltonian plus exactly determined quartic
corrections.
\end{corollary}

\begin{proof}
\emph{Conservation.} Let $\xi \in C^2([a,b])$ satisfy
\eqref{eq:cosh-sinh}, and set $p := m\sinh(\dot\xi)$. Using
$\cosh(\operatorname{arsinh}(p/m)) = \sqrt{1 + p^2/m^2}$, the
Hamiltonian $H_\mathrm{nat}$ rewritten in $(\xi,\dot\xi)$ coordinates
along the trajectory is
\[
  E(t) \;:=\; H_\mathrm{nat}(\xi(t), p(t))
       \;=\; m\bigl[\dot\xi(t)\sinh(\dot\xi(t))
                  - \cosh(\dot\xi(t)) + 1\bigr]
            + k\bigl[\cosh(\xi(t)) - 1\bigr].
\]
Direct differentiation gives
\[
\begin{aligned}
  \dot E(t)
    \;&=\; m\bigl[\ddot\xi\sinh(\dot\xi)
                + \dot\xi\cosh(\dot\xi)\,\ddot\xi
                - \sinh(\dot\xi)\,\ddot\xi\bigr]
       + k\sinh(\xi)\,\dot\xi \\
    \;&=\; \dot\xi\bigl[m\cosh(\dot\xi)\,\ddot\xi + k\sinh(\xi)\bigr]
    \;=\; 0
\end{aligned}
\]
by \eqref{eq:cosh-sinh}; the $\ddot\xi\sinh(\dot\xi)$ and
$-\sinh(\dot\xi)\ddot\xi$ terms cancel exactly. Hence $E$ is constant
on $[a,b]$.

\emph{Expansion.} Combine Proposition~\ref{prop:TH} (giving
$T_H(p) = p^2/(2m) - p^4/(24m^3) + O(p^6/m^5)$) with the
small-$\xi$ Taylor of $\cosh$:
$k[\cosh\xi - 1] = k\xi^2/2 + k\xi^4/24 + O(\xi^6)$.
\end{proof}

\begin{remark}[Joint saddle structure of $L_{\mathrm{nat}}$ and explicit dropout time]\label{rem:saddle}
We anticipate here the conjugate-time machinery of
\S\ref{sec:scope} (Definition~\ref{def:conjugate-time} and
Theorems~\ref{thm:short-time}/\ref{thm:conjugate}) to compute the
dropout time for the native oscillator; only the small-amplitude
Jacobi computation is performed in this remark, and the statements
about strict weak local minimality below are direct specializations
of those theorems to $V = k\Jlog$.

$L_{\mathrm{nat}}$ is \emph{convex} in $\dot\xi$ (from
$\cosh\dot\xi$) but \emph{concave} in $\xi$ (from $-\cosh\xi$). It
is therefore a saddle jointly in $(\xi, \dot\xi)$, and the free
global-minimum theorem of Section~\ref{sec:convexity} does
\emph{not} extend to it.

The dropout is quantitative. Linearize \eqref{eq:cosh-sinh} along
the trivial extremal $\xi_*\equiv 0$: with
$L_{\dot\xi\dot\xi}(0,0) = m$ and
$L_{\xi\xi}(0,0) = -k$, the Jacobi equation
\eqref{eq:jacobi} reduces to the harmonic-oscillator equation
\[
  m\,\ddot\eta + k\,\eta = 0,
\]
whose solution with $\eta(a) = 0$ is
$\eta(t) = C\sin\!\bigl(\sqrt{k/m}\,(t-a)\bigr)$. The first
conjugate time of Definition~\ref{def:conjugate-time} along
$\xi_* \equiv 0$ is therefore
\[
  t_c \;=\; a + \pi\sqrt{m/k}
        \;=\; a + \pi/\omega,
\]
exactly half the harmonic period. Specializing
Theorems~\ref{thm:short-time}/\ref{thm:conjugate} to $V = k\Jlog$
gives three regimes:
\begin{itemize}[noitemsep,topsep=2pt]
  \item \emph{Short-time regime, $b - a < \pi/\omega$.} Then
    $b < t_c$, so by Theorem~\ref{thm:short-time}, $\xi_*\equiv 0$
    is a strict weak local minimizer of $\actionL_V$ on $\mathcal V$.
  \item \emph{Long-time regime, $b - a > \pi/\omega$.} Then
    $t_c = a + \pi/\omega \in (a, b)$ is a conjugate point in the
    open interval, so by Theorem~\ref{thm:conjugate},
    $\xi_*\equiv 0$ is not a strict weak local minimizer.
  \item \emph{Boundary case, $b - a = \pi/\omega$.} Then the first
    conjugate time lies exactly at $t_c = b$. This degenerate case
    falls outside Theorems~\ref{thm:short-time}/\ref{thm:conjugate}
    as stated, since Theorem~\ref{thm:short-time} requires the
    strict inequality $b < t_c$ and Theorem~\ref{thm:conjugate}
    requires $c$ in the open interval $(a, b)$. The behaviour at
    this boundary is governed by the leading nonzero higher-order
    term of $\actionL_V$ along the Jacobi field
    $\eta(t) = \sin(\omega(t-a))$, and is resolved by
    Proposition~\ref{prop:boundary-conjugate} below for
    $k\neq m$ (according to the sign of $k-m$); in the pure case
    $k=m$ the action is exactly flat along the null field, so strict
    local minimality fails and non-strict minimality is left open.
\end{itemize}
Thus the native oscillator transitions from local-minimizer status
to non-minimizer at exactly half the harmonic period, recovering the
classical conjugate-point picture for the small-amplitude limit
of \eqref{eq:cosh-sinh}; the degenerate boundary
$b - a = \pi/\omega$ is settled by
Proposition~\ref{prop:boundary-conjugate}.
\end{remark}

\begin{proposition}[Resolution of the boundary conjugate case]\label{prop:boundary-conjugate}
Let $m,k>0$, $\omega=\sqrt{k/m}$, and let $\xi_*\equiv 0$ be the
trivial extremal of the native Lagrangian $L_{\mathrm{nat}}$
(Definition~\ref{def:native}) on $[a,b]$ at the boundary length
$b-a=\pi/\omega$, so that the first conjugate time lies exactly at
$t_c=b$. Let $\mathcal V$ be the class of $C^1$ endpoint-fixed
variations of Theorem~\ref{thm:short-time}. Then:
\begin{enumerate}[label=\textup{(\roman*)},noitemsep,topsep=2pt]
  \item The second variation
    $\delta^2\actionL_V[\eta]=\int_a^b\bigl(m\dot\eta^2-k\eta^2\bigr)\,dt$
    is positive semidefinite on $H^1_0([a,b])$, with one-dimensional
    null space $\mathrm{span}\{\eta_0\}$,
    $\eta_0(t)=\sin(\omega(t-a))$, and is strictly positive on the
    $L^2$-orthogonal complement of $\eta_0$ with spectral gap
    $\lambda_2=3k>0$.
  \item Along the null direction the leading nonvanishing term of the
    action is quartic,
    \[
      \actionL_V[\varepsilon\eta_0]
        \;=\; \frac{\pi\omega}{64}\,(k-m)\,\varepsilon^4
            \;+\; O(\varepsilon^6).
    \]
  \item Hence, at $b-a=\pi/\omega$: if $k>m$ then $\xi_*\equiv 0$ is
    a strict weak local minimizer of $\actionL_V$ on $\mathcal V$; if
    $k<m$ it is \emph{not} a weak local minimizer (the explicit
    family $\varepsilon\eta_0$ strictly lowers the action); and if
    $k=m$ (the pure d'Alembert case of
    Definition~\ref{def:native}) the action is \emph{exactly flat}
    along the entire family $\varepsilon\eta_0$, so $\xi_*\equiv 0$
    is \emph{not a strict} weak local minimizer; whether it is a
    (non-strict) weak local minimizer requires a mixed-direction
    analysis transverse to $\eta_0$ and is not decided by the
    one-parameter expansion along the null field alone.
\end{enumerate}
\end{proposition}

\begin{proof}
\emph{(i)} Since $L_{\xi\dot\xi}=0$, $L_{\dot\xi\dot\xi}(0,0)=m$ and
$L_{\xi\xi}(0,0)=-V''(0)=-k$, the accessory quadratic form is
$\delta^2\actionL_V[\eta]=\int_a^b(m\dot\eta^2-k\eta^2)\,dt$. The
Dirichlet Sturm--Liouville operator $-m\,d^2/dt^2-k$ on $[a,b]$ has
eigenpairs $\eta_n(t)=\sin\!\bigl(n\pi(t-a)/(b-a)\bigr)$,
$\lambda_n=mn^2\pi^2/(b-a)^2-k$. At $b-a=\pi/\omega$, i.e.\
$(b-a)^2=\pi^2 m/k$, these become $\lambda_n=(n^2-1)k$, so
$\lambda_1=0$ (eigenfunction $\eta_0=\sin(\omega(t-a))$) and
$\lambda_n=(n^2-1)k\ge 3k>0$ for $n\ge 2$. Thus
$\delta^2\actionL_V\ge 0$, vanishes exactly on
$\mathrm{span}\{\eta_0\}$, and for $\zeta=\sum_{n\ge2}c_n\hat\eta_n$
in the $L^2$-orthogonal complement (with $\hat\eta_n$
$L^2$-orthonormal) satisfies
$\delta^2\actionL_V[\zeta]=\sum_{n\ge2}\lambda_n c_n^2\ge
3k\|\zeta\|_{L^2}^2$. Moreover, since
$\|\dot\zeta\|_{L^2}^2=\sum_{n\ge2}\bigl(n\pi/(b-a)\bigr)^2c_n^2
=\sum_{n\ge2}(n^2k/m)\,c_n^2$ at this length and
$(n^2-1)/n^2\ge\tfrac34$ for $n\ge2$, one has the $H^1$-coercivity
\begin{equation}\label{eq:bdry-coercive}
  \delta^2\actionL_V[\zeta]=\sum_{n\ge2}(n^2-1)k\,c_n^2
    \;\ge\;\tfrac34\,m\,\|\dot\zeta\|_{L^2}^2,
  \qquad \zeta\perp_{L^2}\eta_0.
\end{equation}
Equivalently, the semidefiniteness is the equality case of the
sharp Friedrichs inequality (Corollary~\ref{cor:sobolev-gap}) at
$(b-a)^2=\pi^2 m/k$.

\emph{(ii)} With $\xi=\varepsilon\eta_0$,
$\dot\xi=\varepsilon\dot\eta_0$ and
$\cosh u-1=\tfrac12u^2+\tfrac1{24}u^4+O(u^6)$,
\[
  \actionL_V[\varepsilon\eta_0]
    =\frac{\varepsilon^2}{2}\!\int_a^b\!(m\dot\eta_0^2-k\eta_0^2)\,dt
    +\frac{\varepsilon^4}{24}\!\int_a^b\!(m\dot\eta_0^4-k\eta_0^4)\,dt
    +O(\varepsilon^6).
\]
The $\varepsilon^2$ coefficient is
$\tfrac12\delta^2\actionL_V[\eta_0]=0$ by (i), and the cubic term
vanishes since $\cosh$ is even. With $\dot\eta_0=\omega\cos(\omega(t-a))$,
the substitution $s=\omega(t-a)\in[0,\pi]$ and
$\int_0^\pi\sin^4=\int_0^\pi\cos^4=\tfrac{3\pi}{8}$ give
$\int_a^b\eta_0^4\,dt=\tfrac{3\pi}{8\omega}$ and
$\int_a^b\dot\eta_0^4\,dt=\tfrac{3\pi\omega^3}{8}$, whence
\[
  \frac1{24}\int_a^b(m\dot\eta_0^4-k\eta_0^4)\,dt
  =\frac{\pi}{64}\cdot\frac{m\omega^4-k}{\omega}
  =\frac{\pi\omega}{64}\,(k-m),
\]
using $\omega^2=k/m$ so that $m\omega^4-k=k(\omega^2-1)=k(k-m)/m$ and
$k/(m\omega)=\omega$.

\emph{(iii), case $k<m$.} The quartic coefficient
$\tfrac{\pi\omega}{64}(k-m)$ is negative, so
$\actionL_V[\varepsilon\eta_0]<0=\actionL_V[\xi_*]$ for all small
$\varepsilon\ne0$; the admissible one-parameter family
$\varepsilon\eta_0\in\mathcal V$ strictly lowers the action, hence
$\xi_*\equiv0$ is not a weak local minimizer.

\emph{Case $k>m$.} We prove directly that $\actionL_V[\eta]>0$ for
every $\eta\in C^1_0([a,b])$ with $0<\|\eta\|_{C^1}\le\epsilon$ and a
suitable $\epsilon>0$, which is strict weak local minimality (the
competitor path is $\xi_*+\eta=\eta$). Fix $\rho\in(0,1)$ and assume
$\|\eta\|_\infty\le\rho$. Two pointwise all-order $\cosh$ estimates
hold: from
$\cosh u-1-\tfrac12u^2-\tfrac1{24}u^4=\sum_{j\ge3}u^{2j}/(2j)!\ge0$,
\[
  \cosh\dot\eta-1\;\ge\;\tfrac12\dot\eta^2+\tfrac1{24}\dot\eta^4,
\]
while Proposition~\ref{prop:Kkin}(iv) gives, for $|\eta|\le\rho<1$,
\[
  \cosh\eta-1\;\le\;\tfrac12\eta^2+\frac{\eta^4}{24(1-\rho^2)}.
\]
Substituting both into $\actionL_V$, with
$Q[\eta]:=\delta^2\actionL_V[\eta]=\int_a^b(m\dot\eta^2-k\eta^2)\,dt$,
\begin{equation}\label{eq:bdry-lowerbound}
  \actionL_V[\eta]\;\ge\;\tfrac12 Q[\eta]
    +\frac{m}{24}\!\int_a^b\!\dot\eta^4\,dt
    -\frac{k}{24(1-\rho^2)}\!\int_a^b\!\eta^4\,dt.
\end{equation}

Decompose $\eta=c\,\eta_0+\zeta$ with $\eta_0(t)=\sin(\omega(t-a))$
and $\langle\zeta,\eta_0\rangle_{L^2}=0$; both $\eta_0,\zeta\in
C^1_0([a,b])$. Integrating by parts and using
$m\ddot\eta_0=-m\omega^2\eta_0=-k\eta_0$, the $Q$-cross term
vanishes,
$\int_a^b(m\dot\eta_0\dot\zeta-k\eta_0\zeta)\,dt
=\int_a^b(k\eta_0\zeta-k\eta_0\zeta)\,dt=0$, so
$Q[\eta]=Q[\zeta]\ge\tfrac34 m\|\dot\zeta\|_{L^2}^2$ by
\eqref{eq:bdry-coercive}. The splitting is bounded in $C^1$: as
$\|\eta_0\|_{L^2}^2=(b-a)/2$, the coefficient
$c=\langle\eta,\eta_0\rangle_{L^2}/\|\eta_0\|_{L^2}^2$ obeys
$|c|\le\sqrt2\,\|\eta\|_\infty$, whence
$\|\zeta\|_\infty\le(1+\sqrt2)\|\eta\|_\infty$ and
$\|\dot\zeta\|_\infty\le(1+\sqrt2\,\omega)\|\eta\|_{C^1}$.

For the quartics use, for any fixed $\delta\in(0,1)$, the elementary
Young inequalities
$(x+y)^4\ge(1-\delta)x^4-C_\delta y^4$ and
$(x+y)^4\le(1+\delta)x^4+C_\delta y^4$ (with $C_\delta$ depending
only on $\delta$), applied pointwise with
$(x,y)=(c\dot\eta_0,\dot\zeta)$ and $(x,y)=(c\eta_0,\zeta)$ and
integrated:
\[
  \int_a^b\!\dot\eta^4\,dt\ge(1-\delta)c^4\!\int_a^b\!\dot\eta_0^4\,dt
    -C_\delta\!\int_a^b\!\dot\zeta^4\,dt,
  \quad
  \int_a^b\!\eta^4\,dt\le(1+\delta)c^4\!\int_a^b\!\eta_0^4\,dt
    +C_\delta\!\int_a^b\!\zeta^4\,dt.
\]
Inserting these into \eqref{eq:bdry-lowerbound} and using the values
$\tfrac{m}{24}\int_a^b\dot\eta_0^4\,dt=\tfrac{k\pi\omega}{64}$ and
$\tfrac{k}{24}\int_a^b\eta_0^4\,dt=\tfrac{k\pi}{64\omega}$ from (ii)
(recall $m\omega^3=k\omega$), the $c^4$ contribution is
$c^4\,\mu(\delta,\rho)$ with
\[
  \mu(\delta,\rho)
    :=(1-\delta)\frac{k\pi\omega}{64}
      -\frac{1+\delta}{1-\rho^2}\,\frac{k\pi}{64\omega}
    \;\xrightarrow[(\delta,\rho)\to(0,0)]{}\;
      \frac{k\pi}{64}\Bigl(\omega-\tfrac1\omega\Bigr)
      =\frac{\pi\omega}{64}(k-m)>0.
\]
Fix $\delta,\rho$ small enough that
$\mu(\delta,\rho)\ge\tfrac{\pi\omega}{128}(k-m)=:\mu_0>0$. The
remaining $\zeta$-dependent terms are
\[
  \tfrac12 Q[\zeta]
    -\frac{mC_\delta}{24}\!\int_a^b\!\dot\zeta^4\,dt
    -\frac{kC_\delta}{24(1-\rho^2)}\!\int_a^b\!\zeta^4\,dt.
\]
By $\int\dot\zeta^4\le\|\dot\zeta\|_\infty^2\|\dot\zeta\|_{L^2}^2$,
$\int\zeta^4\le\|\zeta\|_\infty^2 C_P^2\|\dot\zeta\|_{L^2}^2$
(Poincar\'e $\|\zeta\|_{L^2}\le C_P\|\dot\zeta\|_{L^2}$), and the
$C^1$ bounds on $\zeta$ above, both negative terms are
$\le C_*\,\|\eta\|_{C^1}^2\,\|\dot\zeta\|_{L^2}^2$ for a constant
$C_*=C_*(\delta,\rho,b-a,\omega,m,k)$. With
$\tfrac12 Q[\zeta]\ge\tfrac38 m\|\dot\zeta\|_{L^2}^2$, the
$\zeta$-terms are
$\ge(\tfrac38 m-C_*\|\eta\|_{C^1}^2)\|\dot\zeta\|_{L^2}^2$. Choosing
$\epsilon\le\rho$ with $C_*\epsilon^2\le\tfrac3{16}m$ gives, for all
$\|\eta\|_{C^1}\le\epsilon$,
\[
  \actionL_V[\eta]\;\ge\;\mu_0\,c^4
    +\tfrac{3}{16}m\,\|\dot\zeta\|_{L^2}^2\;\ge\;0,
\]
with equality only if $c=0$ and $\dot\zeta\equiv0$, i.e.\ (since
$\zeta\in H^1_0$) $\zeta=0$ and so $\eta=0$. Hence
$\actionL_V[\eta]>0=\actionL_V[\xi_*]$ for every
$\eta\in\mathcal V\setminus\{0\}$ with $\|\eta\|_{C^1}\le\epsilon$,
and $\xi_*\equiv0$ is a strict weak local minimizer.

\emph{Case $k=m$.} Here $\omega=1$, the interval has length
$\pi$, and $\eta_0(t)=\sin(t-a)$. Along the one-parameter family
$\xi=\varepsilon\eta_0$ the action is \emph{exactly flat for every
$\varepsilon$}, not merely quartically degenerate: with
$L_{\mathrm{nat}}=m[\cosh\dot\xi-\cosh\xi]$ and the substitution
$s=t-a\in[0,\pi]$,
\[
  \actionL_V[\varepsilon\eta_0]
    = m\int_0^\pi\bigl[\cosh(\varepsilon\cos s)
        - \cosh(\varepsilon\sin s)\bigr]\,ds
    = 0,
\]
because $\int_0^\pi f(\cos s)\,ds=\int_0^\pi f(\sin s)\,ds$ for
every even $f$ (both equal $2\int_0^{\pi/2} f(\sin s)\,ds$),
applied to $f=\cosh(\varepsilon\,\cdot)$. Since
$\varepsilon\eta_0\in\mathcal V$ is admissible and arbitrarily
small in $C^1$ with $\actionL_V[\varepsilon\eta_0]=0
=\actionL_V[\xi_*]$, the strict inequality required for strict weak
local minimality fails. The exact flatness also shows that no
finite-order expansion along $\eta_0$ can decide the question:
whether $\xi_*\equiv 0$ is a non-strict weak local minimizer
depends on the behaviour of $\actionL_V$ in directions transverse
to $\eta_0$ and is left open here.
\end{proof}

\subsection{Newton's law from the general cosh Lagrangian}\label{subsec:general}

\begin{definition}[Cosh Lagrangian with general potential]
Given a potential $V : \R \to \R$ of class $C^1$ (either the
native $V = k\Jlog$ or an external input for general interactions),
and mass coupling $m > 0$, the \emph{cosh Lagrangian} is
\[
  \actionL(\xi, \dot\xi) \;:=\; \Kkin_m(\dot\xi) - V(\xi)
     \;=\; m\bigl[\cosh(\dot\xi) - 1\bigr] - V(\xi),
\]
with associated action
$\actionL_V[\gamma] := \int_a^b \actionL(\xi(t), \dot\xi(t))\,dt$
for admissible paths $\gamma$ with $\xi=\log\gamma$ for which this
integral is defined.
\end{definition}

\begin{theorem}[Newton's second law from the cosh Lagrangian]\label{thm:newton-cosh}
Let $a<b$, $m>0$, and $V\in C^1(\R)$. For $C^2$ curves
$\xi:[a,b]\to\R$, the
Euler--Lagrange equation of $\actionL$ is
\begin{equation}\label{eq:newton-cosh}
  m\cosh(\dot\xi(t))\,\ddot\xi(t) + V'(\xi(t)) \;=\; 0.
\end{equation}
In the small-velocity limit $\cosh(\dot\xi) \to 1$, this reduces to
\[
  m\,\ddot\xi(t) \;=\; -V'(\xi(t)),
\]
Newton's second law with mass $m$ and force $F = -V'$. Under the
kinematic embedding $q := \xi$, this is the standard form
$m\ddot q = -V'(q)$.
\end{theorem}

\begin{proof}
$\partial_{\dot\xi}\actionL = m\sinh(\dot\xi)$, so
$(d/dt)\partial_{\dot\xi}\actionL = m\cosh(\dot\xi)\ddot\xi$.
$\partial_\xi\actionL = -V'(\xi)$.
The Euler--Lagrange equation
$(d/dt)\partial_{\dot\xi}\actionL - \partial_\xi\actionL = 0$ gives
$m\cosh(\dot\xi)\ddot\xi - (-V'(\xi)) = 0$, i.e.,
\eqref{eq:newton-cosh}. The small-velocity limit is immediate.
\end{proof}

\begin{remark}[Newton's first law: inertia]\label{rem:inertia}
For $V'\equiv 0$, equation~\eqref{eq:newton-cosh} reduces to
$\ddot\xi = 0$, i.e., uniform log-velocity motion. This is the
mechanical reading, under the kinematic embedding $q=\xi$, of the
variational statement Theorem~\ref{thm:kinetic-EL}; it is also
equivalent to the conservation of the cosh momentum
$p = m\sinh\dot\xi$ in
Corollary~\ref{thm:momentum} below, via $\dot p = m\cosh(\dot\xi)\ddot\xi$
and $\cosh > 0$. We single it out by name only because the bridge
to Newtonian mechanics motivates labeling, not because it is a
separate fact.
\end{remark}

\section{Hamiltonian Formulation and Conservation Laws}\label{sec:hamilton}

Having established the cosh Lagrangian in Section~\ref{sec:newton},
we now record its Hamiltonian dual and the associated conservation
laws. The Hamiltonian $H = T_H(p) + V(\xi)$ is the foundationally
primary object in the bridge construction
(Definition~\ref{def:additive}): the additively combined total cost
from which the Lagrangian is Legendre-derived. We collect below the
mass-coupled Hamilton equations, the small-momentum expansion, and
energy conservation.

\begin{definition}[Cosh momentum and Hamiltonian, mass-coupled]
Let $m>0$ and let $V:\R\to\R$. For
$\actionL(\xi, \dot\xi) = \Kkin_m(\dot\xi) - V(\xi) =
m[\cosh(\dot\xi) - 1] - V(\xi)$, the \emph{conjugate momentum} is
\[
  p \;:=\; \frac{\partial\actionL}{\partial\dot\xi}
     \;=\; m\sinh(\dot\xi),
  \qquad
  \dot\xi \;=\; \operatorname{arsinh}\!\left(\frac{p}{m}\right),
\]
and the \emph{Hamiltonian} is
\[
  H(\xi, p) \;:=\; p\,\dot\xi - \actionL
    \;=\; T_H(p) + V(\xi)
    \;=\; p\,\operatorname{arsinh}\!\left(\frac{p}{m}\right)
        - \sqrt{m^2 + p^2} + m + V(\xi),
\]
with $T_H$ as in Proposition~\ref{prop:TH} and using the identity
$\cosh(\operatorname{arsinh}(u)) = \sqrt{1 + u^2}$.
\end{definition}

\begin{proposition}[Small-momentum limit of $H$]\label{prop:H-smallp}
Let $m>0$ and $V:\R\to\R$. As $|p|/m \to 0$,
\[
  H(\xi, p) \;=\; \frac{p^2}{2m} + V(\xi)
    \;-\; \frac{p^4}{24\,m^3} + O(p^6/m^5).
\]
Under the kinematic embedding $q := \xi$, this is the standard
non-relativistic Hamiltonian $H = p^2/(2m) + V(q)$ plus a
species-scale-suppressed quartic correction.
\end{proposition}

\begin{proof}
Specialize Proposition~\ref{prop:TH} and add $V(\xi)$.
\end{proof}

\begin{theorem}[Hamilton's equations from EL]\label{thm:hamilton}
Let $a<b$, $m>0$, and let $V\in C^1(\R)$. Let
\[
  H(\xi,p)=p\,\operatorname{arsinh}(p/m)-\sqrt{m^2+p^2}+m+V(\xi).
\]
If $\xi\in C^2([a,b])$ satisfies \eqref{eq:newton-cosh} and
$p(t)=m\sinh(\dot\xi(t))$, then Hamilton's equations hold:
\[
  \dot\xi \;=\; \partial_p H \;=\; \operatorname{arsinh}\!\left(\frac{p}{m}\right),
  \qquad
  \dot p \;=\; -\partial_\xi H \;=\; -V'(\xi).
\]
\end{theorem}

\begin{proof}
The first equation is the inverse Legendre transform
($p = m\sinh(\dot\xi) \iff \dot\xi = \operatorname{arsinh}(p/m)$).
The second follows from \eqref{eq:newton-cosh}:
$\dot p = (d/dt)[m\sinh(\dot\xi)] = m\cosh(\dot\xi)\ddot\xi =
-V'(\xi)$.
\end{proof}

\subsection{Conservation laws from one-parameter symmetries}\label{sec:noether}

Two standard conservation laws we use elsewhere in this work --
energy and the cosh momentum -- arise as Noether charges of
explicit one-parameter symmetries of the action. We take the
Noether derivation as primary; energy and momentum conservation
are then immediate corollaries. This presentation also makes
visible why the cosh momentum is \emph{not} conserved for the
native cosh--sinh oscillator.

\begin{theorem}[Noether charges of the cosh Lagrangian]\label{thm:noether-explicit}
Let $L(\xi,\dot\xi) = m\bigl[\cosh(\dot\xi) - 1\bigr] - V(\xi)$
with $m > 0$ and $V \in C^1(\R)$.

\smallskip
\noindent\textbf{(a) Time-translation symmetry.} The one-parameter
group $\Phi^\alpha_t : (t,\xi) \mapsto (t + \alpha, \xi)$ leaves
$L$ invariant for every $\alpha \in \R$, since $L$ has no explicit
$t$-dependence. The associated Noether charge is the energy
\[
  E \;=\; \dot\xi\,\partial_{\dot\xi}L \;-\; L
    \;=\; m\bigl[\dot\xi\sinh(\dot\xi) - \cosh(\dot\xi) + 1\bigr]
        + V(\xi),
\]
and it is conserved on every $C^2$ solution of
\eqref{eq:newton-cosh}.

\smallskip
\noindent\textbf{(b) $\xi$-translation symmetry, free case.} If
$V$ is constant on $\R$, the one-parameter group
$\Phi^\alpha_\xi : (t,\xi) \mapsto (t, \xi + \alpha)$ leaves $L$
invariant. The associated Noether charge is the cosh momentum
\(
  p \;=\; \partial_{\dot\xi}L \;=\; m\sinh(\dot\xi),
\)
conserved on every $C^2$ solution of \eqref{eq:newton-cosh}.

\smallskip
\noindent\textbf{(c) Native oscillator: no $\xi$-translation
symmetry.} For $V = k\Jlog$ with $k > 0$, the native potential
satisfies $\Jlog(\xi + \alpha) = \cosh(\xi+\alpha) - 1
\not\equiv \Jlog(\xi)$ whenever $\alpha \ne 0$. Hence $\Phi^\alpha_\xi$
is \emph{not} a Noether symmetry of $L_{\mathrm{nat}}$, and the cosh
momentum is not conserved along solutions of \eqref{eq:cosh-sinh};
only the energy of Corollary~\ref{cor:native-energy} is.
\end{theorem}

\begin{proof}
\emph{(a)} For an autonomous first-order Lagrangian, the Noether
identity for time-translation gives $dE/dt = \partial_t L = 0$
along solutions. Explicitly, with
$E = m[\dot\xi\sinh(\dot\xi) - \cosh(\dot\xi) + 1] + V(\xi)$,
\[
  \dot E
    = m\bigl[\ddot\xi\sinh(\dot\xi) + \dot\xi\cosh(\dot\xi)\ddot\xi
              - \sinh(\dot\xi)\ddot\xi\bigr] + V'(\xi)\dot\xi
    = \dot\xi\bigl[m\cosh(\dot\xi)\ddot\xi + V'(\xi)\bigr]
    = 0
\]
by \eqref{eq:newton-cosh}.

\emph{(b)} The infinitesimal generator is
$(\delta t, \delta\xi) = (0, 1)$. Under this generator,
$\delta L = \partial_\xi L \cdot 1 = -V'(\xi)$, which vanishes
identically iff $V$ is constant. The Noether charge is
$\partial_{\dot\xi}L \cdot \delta\xi = m\sinh(\dot\xi)$, and
direct differentiation gives
\[
  \dot p \;=\; m\cosh(\dot\xi)\ddot\xi
        \;=\; -V'(\xi) \;=\; 0
\]
by \eqref{eq:newton-cosh} and $V'\equiv 0$.

\emph{(c)} We use the \emph{quasi-invariance} (Bessel-Hagen) form of
Noether's theorem~\cite{noether1918,besselhagen1921,olver1993}:
$\Phi^\alpha_\xi$ is a Noether symmetry of
$L_{\mathrm{nat}}$ iff there exists a function $F=F(t,\xi,\dot\xi)$
of class $C^1$ such that
$\delta L_{\mathrm{nat}} = dF/dt$ along arbitrary $C^1$ curves, not
merely along solutions. Here
$\delta L_{\mathrm{nat}} = \partial_\xi L_{\mathrm{nat}}\cdot 1
= -k\sinh(\xi)$, a function of $\xi$ alone. We show that
$-k\sinh(\xi)$ is not a total $t$-derivative of any
$F(t,\xi,\dot\xi)$.

Suppose for contradiction that
$-k\sinh(\xi) = dF/dt = \partial_t F + F_\xi\,\dot\xi + F_{\dot\xi}\,\ddot\xi$
identically in $(t,\xi,\dot\xi,\ddot\xi)$. The left-hand side is
independent of $\dot\xi$ and $\ddot\xi$, so equating coefficients of
$\ddot\xi$ and of $\dot\xi$ gives $F_{\dot\xi}\equiv 0$ and
$F_\xi\equiv 0$, hence $F=F(t)$. Then
$dF/dt = F'(t)$ depends only on $t$, but $-k\sinh(\xi)$ depends
nontrivially on $\xi$ (since $k>0$), contradiction. Hence no such
$F$ exists, $\Phi^\alpha_\xi$ is not a quasi-invariance of
$L_{\mathrm{nat}}$, and Noether's theorem produces no conserved
charge from $\xi$-translation for the native potential.

\emph{Direct verification of non-conservation.} Along any solution
of \eqref{eq:cosh-sinh},
$\dot p = m\cosh(\dot\xi)\ddot\xi = -k\sinh(\xi)$, which vanishes
identically only at the trivial extremal $\xi \equiv 0$. So even
without the symmetry analysis, $p$ is manifestly not conserved.
\end{proof}

\begin{corollary}[Energy conservation]\label{thm:energy}
Let $a<b$, $m>0$, $V\in C^1(\R)$, and let $H$ be as in
Theorem~\ref{thm:hamilton}. If $\xi\in C^2([a,b])$ satisfies
\eqref{eq:newton-cosh} and $p(t)=m\sinh(\dot\xi(t))$, then the
total energy $E(t):=H(\xi(t),p(t))$ is constant on $[a,b]$.
\end{corollary}

\begin{proof}
Theorem~\ref{thm:noether-explicit}(a).
\end{proof}

\begin{corollary}[Momentum conservation in the free case]\label{thm:momentum}
Let $a<b$, $m>0$, and let $V'\equiv 0$ on $\R$. If
$\xi\in C^2([a,b])$ satisfies \eqref{eq:newton-cosh}, then the
cosh momentum $p(t) := m\sinh(\dot\xi(t))$ is constant on $[a,b]$.
Equivalently, $\ddot\xi \equiv 0$ (Remark~\ref{rem:inertia} /
Theorem~\ref{thm:kinetic-EL}), since $\dot p = m\cosh(\dot\xi)\ddot\xi$
and $\cosh > 0$.
\end{corollary}

\begin{proof}
Theorem~\ref{thm:noether-explicit}(b); the equivalent ``$\ddot\xi=0$''
form is immediate from $\dot p = m\cosh(\dot\xi)\ddot\xi = 0$ and
$\cosh > 0$.
\end{proof}

\begin{remark}[Symmetry inventory of the d'Alembert framework]\label{rem:symmetry-inventory}
Within the cosh Lagrangian $L = m[\cosh\dot\xi - 1] - V(\xi)$, the
maximal one-parameter symmetry group is generated by:
\begin{itemize}[noitemsep,topsep=2pt]
  \item time translation, always present, conserving energy;
  \item $\xi$-translation, present iff $V$ is constant on $\R$,
    conserving the cosh momentum $m\sinh\dot\xi$;
  \item the discrete reflection $\xi \mapsto -\xi$ (with
    $\dot\xi \mapsto -\dot\xi$), present iff $V$ is even, which is
    a $\Z/2\Z$ symmetry rather than a one-parameter group and so
    yields no continuous Noether charge.
\end{itemize}
The native potential $V = k\Jlog$ is even but not translation
invariant, so it admits the reflection symmetry but not a continuous
$\xi$-translation symmetry; consequently the only continuous
conservation law of the native oscillator is energy. This explains
both \emph{why} energy survives the addition of the native potential
and \emph{why} the cosh momentum does not.
\end{remark}

\section{Potentials and Classical Variational Scope}\label{sec:scope}

Having shown in Section~\ref{sec:newton} how Newton's law arises
in the small-velocity limit of the cosh Lagrangian, we now ask
what variational status the resulting extremals carry when a
non-affine convex potential is added. Our convexity argument in
Section~\ref{sec:convexity} is \emph{free} of any potential;
adding $-V(\xi)$ breaks joint convexity of $\actionL$ in
$(\xi,\dot\xi)$ whenever $V$ is strictly convex (the native case
$V=k\Jlog$ is the basic example). The Lagrangian is then a
saddle: convex in velocity and concave in position. In this
setting the convex global-minimizer theorem of
Theorem~\ref{thm:pla} no longer applies, and we adopt the
classical local-minimum / stationary-action picture as the correct
replacement.

\begin{remark}[Convention on regularity classes for this section]\label{rem:regularity-shift}
\emph{This convention is in force throughout
Section~\ref{sec:scope}.} The free-sector results of
Sections~\ref{sec:pathspace}--\ref{sec:convexity} are stated for
\emph{kinetically admissible} paths in the sense of
Definition~\ref{def:kin-adm}: absolutely continuous positive paths
$\gamma:[a,b]\to\Rplus$ with $\log$-derivative in $L^2([a,b])$ and
finite kinetic action $\int(\cosh\dot\xi - 1)\,dt<\infty$.

All theorems below (Theorems~\ref{thm:short-time}, \ref{thm:stationarity},
\ref{thm:conjugate}) instead require:
\begin{enumerate}[label=(R\arabic*),noitemsep,topsep=2pt]
  \item the potential satisfies $V\in C^2(\R)$ (or $V\in C^1(\R)$ for
    the bare stationarity theorem
    Theorem~\ref{thm:stationarity});
  \item the extremal $\xi_*:[a,b]\to\R$ is of class $C^2([a,b])$ and
    satisfies the strong Euler--Lagrange equation
    \eqref{eq:newton-cosh} pointwise;
  \item variations $\eta:[a,b]\to\R$ are of class $C^1([a,b])$ with
    $\eta(a)=\eta(b)=0$.
\end{enumerate}
This is a deliberate shift to stronger regularity than the
free-sector path space of Section~\ref{sec:pathspace}, made because
the Jacobi sufficient/necessary conditions of Gelfand--Fomin
\cite{gelfand-fomin} are formulated at this regularity. The
free-sector theorems of Section~\ref{sec:convexity} are
\emph{not} re-stated under (R1)--(R3); they remain valid in the
larger kinetically admissible class. The two regularity worlds meet
at $C^2$ uniform-log-velocity paths, which simultaneously satisfy
the convex chord characterization of
Theorem~\ref{thm:pla} and the pointwise Euler--Lagrange equation of
Theorem~\ref{thm:kinetic-EL} (see Remark~\ref{rem:EL-class}).
\end{remark}

\subsection*{Second variation, Jacobi fields, and conjugate times}

Fix $a<b$, $m>0$, and $V\in C^2(\R)$. For a $C^2$ curve
$\xi:[a,b]\to\R$, write
\[
  \actionL_V[\xi]
    \;:=\; \int_a^b \bigl[m(\cosh(\dot\xi(t)) - 1) - V(\xi(t))\bigr]\,dt.
\]
For an admissible path $\gamma$ with $\xi=\log\gamma$, this
agrees with the earlier notation $\actionL_V[\gamma]=\actionL_V[\xi]$. Let
$\xi_*\in C^2([a,b])$ be a solution of the Euler--Lagrange equation
\eqref{eq:newton-cosh}. The classical second-variation theory for
first-order $C^2$ Lagrangians associates to $\xi_*$ a Jacobi equation.
In the present case $L_{\xi\dot\xi}=0$, so the Jacobi equation is
\begin{equation}\label{eq:jacobi}
  \frac{d}{dt}\Bigl(m\cosh(\dot\xi_*(t))\,\dot\eta(t)\Bigr)
    + V''(\xi_*(t))\,\eta(t) \;=\; 0.
\end{equation}

\begin{definition}[Conjugate time]\label{def:conjugate-time}
A time $c\in(a,b]$ is \emph{conjugate to $a$ along $\xi_*$} if there
exists a nonzero $C^2$ solution $\eta$ of \eqref{eq:jacobi} with
$\eta(a)=0$ and $\eta(c)=0$. The \emph{first conjugate time on
$[a,b]$} is the infimum of such $c$ in $(a,b]$, if any exists, and is
defined to be $\infty$ if no conjugate time occurs in $(a,b]$.
\end{definition}

\begin{theorem}[Short-time local minimality]\label{thm:short-time}
Let $a<b$, $m>0$, $V \in C^2(\R)$, and let $\actionL(\xi,\dot\xi) =
\Kkin_m(\dot\xi) - V(\xi)$. Let $\xi_*\in C^2([a,b])$ be a solution of
\eqref{eq:newton-cosh} with fixed endpoints, and let $t_c$ be the first
conjugate time of Definition~\ref{def:conjugate-time}. Denote by
$\mathcal V$ the class of $C^1$ endpoint-fixed variations
\[
  \mathcal V := \{\xi_* + \eta : \eta \in C^1([a,b]),\ \eta(a)=\eta(b)=0\}.
\]
If $b < t_c$ (equivalently, there is no conjugate point in $(a,b]$),
then $\xi_*$ is a strict local minimizer of $\actionL_V$ within
$\mathcal V$, in a sufficiently small $C^1$ neighborhood of $\xi_*$.
\end{theorem}

\begin{proof}
We invoke the Jacobi sufficient-condition theorem in the form of
Gelfand--Fomin
\cite[Ch.~5, Sufficient Conditions for a Weak Extremum]{gelfand-fomin}.
Its hypotheses hold here: the Lagrangian
$L(\xi,\dot\xi)=m(\cosh\dot\xi-1)-V(\xi)$ is $C^2$, the strong
Legendre condition $L_{\dot\xi\dot\xi}=m\cosh(\dot\xi)\ge m>0$ holds,
$\xi_*$ is a $C^2$ extremal, and $b<t_c$ says that no nontrivial
Jacobi field vanishing at $a$ has a second zero in $(a,b]$. The
classical theorem therefore gives strict weak local minimality in a
sufficiently small $C^1$ neighborhood.
\end{proof}

\begin{theorem}[Long-time stationarity]\label{thm:stationarity}
Let $a<b$, $m>0$, $V\in C^1(\R)$, and let
$\xi_*\in C^2([a,b])$ be a solution of \eqref{eq:newton-cosh} with
fixed endpoints. Then the first variation of $\actionL_V$ at
$\xi_*$ vanishes on every $C^1$ endpoint-fixed variation $\eta$,
without any claim of local or global minimality.
\end{theorem}

\begin{proof}
For $\eta\in C^1([a,b])$ with $\eta(a)=\eta(b)=0$, the standard
first-variation formula applies: the kinetic term is smooth in
$\dot\xi$, $V\in C^1$, and $\xi_*\in C^2$, so differentiating under
the integral and integrating the kinetic term by parts gives
\[
  \delta\actionL_V[\xi_*](\eta)
    = \int_a^b \Bigl(\frac{d}{dt}\partial_{\dot\xi}L(\xi_*,\dot\xi_*)
          - \partial_\xi L(\xi_*,\dot\xi_*)\Bigr)\eta(t)\,dt.
\]
Since $\xi_*$ satisfies the Euler--Lagrange equation, this integral
vanishes for all endpoint-fixed $\eta$.
\end{proof}

\begin{theorem}[Conjugate-time obstruction to local minimality]\label{thm:conjugate}
Let $a<b$, $m>0$, $V\in C^2(\R)$, and let
$\xi_*\in C^2([a,b])$ be a solution of \eqref{eq:newton-cosh} with
fixed endpoints. If there exists a time
$c\in(a,b)$ conjugate to $a$ along $\xi_*$
(Definition~\ref{def:conjugate-time}), then $\xi_*$ is not a strict
local minimizer of $\actionL_V$ on the class $\mathcal V$ of $C^1$
endpoint-fixed variations.
\end{theorem}

\begin{proof}
The Lagrangian is $C^2$ and satisfies the strong Legendre condition
$L_{\dot\xi\dot\xi}=m\cosh(\dot\xi)\ge m>0$. The classical Jacobi
necessary condition for a strict weak local minimum therefore applies
to the extremal $\xi_*$. That theorem states that a strict weak local
minimum cannot have a conjugate point to the initial endpoint in the
open interval $(a,b)$; see Gelfand--Fomin
\cite[Ch.~5, Necessary Conditions; Ch.~6]{gelfand-fomin}. The
existence of such a $c$ contradicts the necessary condition, so
$\xi_*$ is not a strict local minimizer on $\mathcal V$.
\end{proof}

Thus the variational principle changes form when a non-affine
strictly convex potential is introduced: the free convex global
minimum is replaced by short-time local minimality up to conjugate
points and long-time stationarity only.

\section{Discussion}\label{sec:discussion}

\subsection{Summary and scope}

Our main result is a free-sector variational theorem on positive
paths. The calibrated d'Alembert equation forces the cosh cost,
and applying that cost to the log-velocity yields the kinetic
action
$\actionA[\gamma] = \int_a^b [\cosh(\dot\xi(t)) - 1]\,dt$ with
$\xi = \log\gamma$. On the kinetically admissible class,
$\actionA$ is \emph{strongly} convex under geometric interpolation
(Theorem~\ref{thm:actionA-convex}), with an explicit $L^2$ slack
of $1$-strong convexity. The chord condition of
Theorem~\ref{thm:pla} therefore implies global minimality with a
quantitative gap (Theorem~\ref{thm:local-global}), and
Corollary~\ref{cor:uniform} identifies the unique fixed-endpoint
minimizer explicitly as the uniform-log-velocity path. The action
gap admits an exact Pythagorean / Bregman identity
(Theorem~\ref{thm:pythagorean}) together with the
Friedrichs--Poincar\'e bound
$\actionA[\gamma]-\actionA[\gamma_*]\ge
\frac{\pi^2}{2(b-a)^2}\|\log(\gamma/\gamma_*)\|_{L^2}^2$
(Corollary~\ref{cor:sobolev-gap}), and the minimum-action profile
$\actionA_*(T,\Delta) = T(\cosh(\Delta/T)-1)$ is the perspective
transform of $\cosh - 1$, jointly convex and positively
$1$-homogeneous, giving the geodesic-concatenation inequality of
Corollary~\ref{cor:concat}. These results form a
Bregman / dually-flat interpretation of the d'Alembert log-cost $\Jlog$ in
the additive coordinate $\xi=\log x$ (metric $\cosh\xi\,d\xi^2$),
which is distinct from the Hessian metric $g_J=x^{-3}\,dx^2$ of
$\Jcost$ in the coordinate $x$ (Remark~\ref{rem:dually-flat}).

\medskip
\noindent\emph{On the nature of the contribution.} The free-sector
arguments are deliberately elementary. Once d'Alembert fixes the
cosh cost (Theorem~\ref{thm:dAlembert-classification}) and
Postulate~\ref{post:step} places it at the log-velocity, the entire
convexity package follows from pointwise convexity of $\cosh$ after
the log change of coordinates
(Remark~\ref{rem:convexity-mechanism}), Jensen's inequality
(Corollary~\ref{cor:uniform}), the perspective construction
(Proposition~\ref{prop:perspective}), the one-dimensional
Friedrichs inequality (Corollary~\ref{cor:sobolev-gap}), and
textbook Bregman / dually-flat geometry
(Remark~\ref{rem:dually-flat}); Theorem~\ref{thm:pla} is the
specialization of the standard first-order optimality criterion for
convex functionals (Remark~\ref{rem:pla-meaning}). The novelty we
claim is therefore one of \emph{provenance and organization} --
that a functional equation together with a single named modeling
postulate forces a globally (not merely locally) minimizing free
action with an exact Bregman gap and a closed-form geodesic
minimizer -- rather than of analytic depth. Correspondingly, the
dynamical content actually forced by d'Alembert is narrow: free
motion is the uniform-log-velocity geodesic, and the interacting
system selected within this framework is the native cosh--sinh
oscillator (the sharper scope statement below); all other potentials
enter as external inputs.

This conclusion is a free-sector result; it does not assert
global minimality for Lagrangians with non-affine strictly convex
potentials. In that case the Lagrangian becomes a saddle in
$(\xi,\dot\xi)$, and Section~\ref{sec:scope} records the
classical replacement: short-time local minimality, long-time
stationarity, and conjugate-time obstructions. The bridge to
Newtonian mechanics is also conditional. It requires the
kinematic identification $q=\xi$, a mass coupling, a
dimensionless time calibration, and the Hamiltonian-primary
Legendre structure (Section~\ref{sec:newton}). With those choices
the cosh action has the Newtonian small-step limit and the
cosh-dual Hamiltonian reduces to $p^2/(2m)+V(q)$ at small
momentum. These bridge statements are interpretive consequences
of additional structure, not part of the free convexity theorem
itself.

A sharper version of the scope statement is the following. Among
all Lagrangians of the form
$\actionL(\xi,\dot\xi) = \Kkin_m(\dot\xi) - V(\xi)$ with $m > 0$,
the \emph{only one forced by the d'Alembert calibration} is
the native cosh--sinh Lagrangian
$L_{\mathrm{nat}}(\xi,\dot\xi) = m[\cosh\dot\xi - 1] - k[\cosh\xi - 1]$
of Definition~\ref{def:native}, in which the same $\cosh - 1$
function appears in both the kinetic and the static slot under
the same d'Alembert uniqueness. Every other potential --
harmonic, Coulomb, gravitational, polynomial, lattice, periodic,
$\dots$ -- is an \emph{external input} from the cost-field
environment that we do not derive from d'Alembert's equation
here. The mathematically privileged dynamical system is therefore
the cosh--sinh oscillator carrying two species couplings
$(m, k)$; even there, the binding coupling $k > 0$ is left as an
unconstrained input, with its provenance (like that of $m$)
lying outside the scope of this paper. For
any external $V$, the results of
Sections~\ref{subsec:general}--\ref{sec:scope} (Newton's law in
the small-velocity limit, Hamilton's equations, energy
conservation, short-time local min / long-time stationarity /
conjugate-time obstruction) apply, but they describe a
\emph{compatible} embedding of an externally specified mechanical
system into the cosh kinetic framework, not a derivation of that
mechanical system from d'Alembert.

\subsection{Open directions}

The most direct extensions are mathematical, and the Bregman /
dually flat reading of \S\ref{subsec:duallyflat} makes several of
them concrete.

\begin{itemize}[noitemsep,topsep=2pt]
  \item \textbf{Multi-component / matrix lift.} On the cone
    $\Rplus^n$ the cost $\Jcost_n(x) = \sum_i\Jcost(x_i)$ inherits
    the Hessian metric $\sum_i x_i^{-3}dx_i^2$, geometric
    interpolation acts componentwise, and
    Theorems~\ref{thm:actionA-convex}--\ref{thm:pythagorean} extend
    verbatim. The non-trivial lift is to $\mathrm{Sym}_{>0}^n$ via
    $\Jcost(X) = \tfrac12\mathrm{tr}(X + X^{-1}) - n$, where the
    affine-invariant geodesic structure suggests comparisons with
    log-Euclidean / Bures--Wasserstein geometry
    \cite{arsigny2007,bhatia_jain_lim2019}.
  \item \textbf{Discrete-time actions.} The discrete kinetic
    action $\sum_k(\cosh(\xi_{k+1}-\xi_k)-1)$ inherits geometric-
    interpolation convexity and the Pythagorean identity at the
    discrete level, giving a discrete optimal-control statement
    parallel to Corollary~\ref{cor:uniform}.
  \item \textbf{Probabilistic identification.} The function
    $\Kkin(v) = \cosh(v) - 1$ is, up to normalization, the cumulant
    generating function of the Skellam$(\tfrac12,\tfrac12)$
    distribution~\cite{skellam1946}. This suggests that $\actionA$ is the rate
    functional of a continuous-time Skellam-type random walk on
    $\R$, with the geodesic of Corollary~\ref{cor:uniform}
    appearing as the corresponding Schr\"odinger
    bridge~\cite{leonard2014}.
  \item \textbf{Field-theoretic cosh-Dirichlet energy.} For
    $\gamma : M \to \Rplus$ on a Riemannian manifold $(M,g)$,
    \(
      \mathcal{E}_M[\gamma] = \int_M(\cosh|\nabla\log\gamma|_g - 1)\,d\mathrm{vol}_g
    \)
    is convex along log-affine homotopies; its Euler--Lagrange
    equation is a degenerate-elliptic cosh-Laplacian with growth
    interpolating between $|\nabla|^2$ and $e^{|\nabla|}$.
  \item \textbf{Comparison with the Hessian energy $\Ehess$.} The
    relation between the Bregman geometry of $\actionA$
    (\S\ref{subsec:duallyflat}) and the Otto-type geometry of
    $\Ehess$ on the same base
    manifold $(\Rplus,\Jcost)$~\cite{otto2001,ambrosio_gigli_savare2008,villani2009}
    deserves explicit treatment: the
    two functionals differ in connection rather than in carrier.
\end{itemize}
These questions preserve our main theme: how much variational
structure is already forced by the algebra of the underlying cost.

\appendix

\section{Hessian Riemannian path-energy on \texorpdfstring{$(\Rplus,g_J)$}{(R>0, gJ)}}\label{app:hessian}

We collect here the structural facts about the Hessian path-energy
\[
  \Ehess[\gamma] \;=\; \int_a^b \tfrac{1}{2}\,\frac{\dot\gamma(t)^2}{\gamma(t)^3}\,dt
\]
of \S\ref{subsec:ground-state} that we did not need for the
free-sector convexity story of Section~\ref{sec:convexity}. The
material here is included for completeness and to make the
$\actionA$-vs-$\Ehess$ distinction unambiguous; we use none of it
in the main theorems of this work. A fuller variational analysis
of $\Ehess$, including any connection to Otto-type
Wasserstein-gradient-flow geometry~\cite{otto2001,villani2009} on a related space of
probability measures, is left to future work; we do not invoke Otto
calculus here.

\subsection{Explicit geodesic family}

\begin{theorem}[Explicit geodesic family]\label{thm:geodesic-family}
The family $\gamma(t) = (at + b)^{-2}$ with $a \neq 0$ and
$at + b > 0$ satisfies the geodesic equation \eqref{eq:geodesic}.
\end{theorem}

\begin{proof}
With $u = at + b$: $\gamma = u^{-2}$, $\dot\gamma = -2au^{-3}$,
$\ddot\gamma = 6a^2 u^{-4}$, $\Gamma(\gamma) = -\tfrac32 u^2$, and
$\ddot\gamma + \Gamma(\gamma)\dot\gamma^2 = 6a^2u^{-4} - 6a^2u^{-4}
= 0$.
\end{proof}

\subsection{Geodesic completeness}

\begin{remark}[Geodesic completeness of $(\Rplus, g)$ at $0$ but not at $\infty$]\label{rem:hessian-completeness}
The Riemannian distance element of $g(x) = x^{-3}$ is
$ds = \sqrt{g(x)}\,dx = x^{-3/2}\,dx$. Direct integration gives
\[
  \int_0^1 x^{-3/2}\,dx \;=\; +\infty, \qquad
  \int_1^\infty x^{-3/2}\,dx \;=\; 2.
\]
Thus the boundary $x=0$ is at infinite Riemannian distance from
any interior point (the metric is complete at $0$), while
$x=\infty$ is at finite Riemannian distance (the metric is
incomplete at $\infty$). Consistently, the geodesic family
$\gamma(t) = (at+b)^{-2}$ with $a<0$ reaches $\gamma=+\infty$ at
the finite parameter $t = -b/a$, whereas $\gamma\to 0^+$ requires
$|t|\to\infty$.

This (in)completeness is a property of the Hessian path-energy
$\Ehess$ and its underlying metric $g$, \emph{not} of the
kinetic action $\actionA$, which uses the log-affine
($e$-)connection rather than the Riemannian connection of $g$.
Our free-sector theorems of Section~\ref{sec:convexity} do not
depend on completeness of $g$. A geodesically complete
continuation of $(\Rplus,g)$ at $\infty$ is left to future work.
\end{remark}

\subsection{Structural comparison of \texorpdfstring{$\actionA$}{A} and \texorpdfstring{$\Ehess$}{E_Hess}}

\begin{remark}[$\actionA$ and $\Ehess$ are distinct functionals]\label{rem:A-vs-E}
The kinetic Euler--Lagrange equation
(Theorem~\ref{thm:kinetic-EL}) selects uniform-log-velocity paths
$\gamma(t) = x_a\,e^{\dot\xi_0(t-a)}$. The Hessian-metric geodesic
equation \eqref{eq:geodesic} selects the power-law family
$\gamma(t) = (at+b)^{-2}$ together with the limiting positive
constant geodesics. As \emph{parameterized} curves $t\mapsto\gamma(t)$
these are different families on $(0,\infty)$, though both include
every positive constant path as a trivial solution. We stress that
the distinction is one of parameterization, not of trajectory: on the
$1$-manifold $\Rplus$ any two fixed-endpoint geodesics share the same
image -- the arc between the endpoints -- since a $1$-dimensional
manifold admits no choice of route, and a connection on it carries no
curvature and only fixes which time-law $t\mapsto\gamma(t)$ is
affine/constant-speed. Thus the exponential and power-law solutions
above are reparameterizations of one another as curves; what differs
is the time-law each connection selects as its geodesic
parameterization. The additional static cost-rate condition for
$\actionJ$ selects only $\gamma\equiv1$, so all three critical
conditions agree precisely at the normalized ground state.
Accordingly $\actionA$ and $\Ehess$ are genuinely different
functionals -- they score parameterizations differently and pick out
different extremal time-laws -- rather than geometrically distinct
trajectories in $\Rplus$.

The kinetic action $\actionA$ is the d'Alembert-calibrated object
-- $\Jcost$ applied to the infinitesimal \emph{step} -- and the
free-sector convexity theorems of Section~\ref{sec:convexity}
apply to it. The Hessian path-energy $\Ehess$ is a separate
Riemannian path-energy whose natural notion of convexity is
\emph{geodesic convexity} along the geodesics of the underlying
metric $g(x) = x^{-3}$ -- a different notion from the
geometric-interpolation (log-affine) convexity of $\actionA$. The
two pictures are complementary rather than equivalent: $\actionA$
formalizes the action principle from the d'Alembert side, while
$\Ehess$ formalizes the Riemannian geometry of the choice manifold
from the Hessian side. We use neither $\Ehess$ nor its geodesic
structure in our main theorems; a fuller development of $\Ehess$'s
geodesic and Otto-Wasserstein structure is left to future work.
\end{remark}



\begin{thebibliography}{99}

\bibitem{arnold1989}
V.~I.~Arnold.
\newblock \emph{Mathematical Methods of Classical Mechanics}.
\newblock Graduate Texts in Mathematics, vol.~60. Translated by K.~Vogtmann and
  A.~Weinstein. Springer-Verlag, New York, 2nd edition, 1989.

\bibitem{pardoguerra_ledger2026}
Pardo-Guerra, S.; Thapa, A.; Simons, M.; Washburn, J. Coherent Comparison as Information Cost: Axiomatic Foundations for Discrete Ledger Dynamics. 
Foundations 2026, 6, 17. https://doi.org/10.3390/foundations6020017

\bibitem{washburn_zlatanovic2026}
Washburn, J.; Zlatanović, M. Uniqueness of the Canonical Reciprocal Cost. 
Mathematics 2026, 14, 935. https://doi.org/10.3390/math14060935

\bibitem{washburn_zlatanovic_beltracchi2026}
Washburn, J.; Zlatanović, M.; Beltracchi, P. Multidimensional Cost Geometry. 
Axioms 2026, 15, 378. https://doi.org/10.3390/axioms15050378

\bibitem{aczel1966}
J.~Acz\'el.
\newblock \emph{Lectures on Functional Equations and Their Applications}.
\newblock Academic Press, New York, 1966.

\bibitem{stetkaer2013}
H.~Stetk\ae r.
\newblock \emph{Functional Equations on Groups}.
\newblock World Scientific, Singapore, 2013.

\bibitem{bregman1967}
L.~M.~Bregman.
\newblock The relaxation method of finding the common point of convex sets and
  its application to the solution of problems in convex programming.
\newblock \emph{USSR Computational Mathematics and Mathematical Physics},
  7(3):200--217, 1967.

\bibitem{amari_nagaoka2000}
S.~Amari and H.~Nagaoka.
\newblock \emph{Methods of Information Geometry}.
\newblock Translations of Mathematical Monographs, vol.~191. American
  Mathematical Society, Providence, RI, and Oxford University Press, Oxford,
  2000.

\bibitem{shima2007}
H.~Shima.
\newblock \emph{The Geometry of Hessian Structures}.
\newblock World Scientific, Hackensack, NJ, 2007.

\bibitem{rockafellar1970}
R.~T.~Rockafellar.
\newblock \emph{Convex Analysis}.
\newblock Princeton Mathematical Series, vol.~28. Princeton University Press,
  Princeton, NJ, 1970.

\bibitem{boyd_vandenberghe2004}
S.~Boyd and L.~Vandenberghe.
\newblock \emph{Convex Optimization}.
\newblock Cambridge University Press, Cambridge, 2004.

\bibitem{hardy_littlewood_polya1952}
G.~H.~Hardy, J.~E.~Littlewood, and G.~P\'olya.
\newblock \emph{Inequalities}.
\newblock Cambridge University Press, Cambridge, 2nd edition, 1952.

\bibitem{noether1918}
E.~Noether.
\newblock Invariante Variationsprobleme.
\newblock \emph{Nachrichten von der K\"oniglichen Gesellschaft der
  Wissenschaften zu G\"ottingen, Mathematisch-physikalische Klasse},
  235--257, 1918.

\bibitem{besselhagen1921}
E.~Bessel-Hagen.
\newblock \"Uber die Erhaltungss\"atze der Elektrodynamik.
\newblock \emph{Mathematische Annalen}, 84(3--4):258--276, 1921.

\bibitem{olver1993}
P.~J.~Olver.
\newblock \emph{Applications of Lie Groups to Differential Equations}.
\newblock Graduate Texts in Mathematics, vol.~107. Springer-Verlag, New York,
  2nd edition, 1993.

\bibitem{gelfand-fomin}
I.~M.~Gelfand and S.~V.~Fomin.
\newblock \emph{Calculus of Variations}.
\newblock Translated and edited by R.~A.~Silverman. Prentice-Hall,
  Englewood Cliffs, NJ, 1963; Dover reprint, 2000.

\bibitem{arsigny2007}
V.~Arsigny, P.~Fillard, X.~Pennec, and N.~Ayache.
\newblock Geometric means in a novel vector space structure on symmetric
  positive-definite matrices.
\newblock \emph{SIAM Journal on Matrix Analysis and Applications},
  29(1):328--347, 2007.

\bibitem{bhatia_jain_lim2019}
R.~Bhatia, T.~Jain, and Y.~Lim.
\newblock On the Bures--Wasserstein distance between positive definite
  matrices.
\newblock \emph{Expositiones Mathematicae}, 37(2):165--191, 2019.

\bibitem{skellam1946}
J.~G.~Skellam.
\newblock The frequency distribution of the difference between two Poisson
  variates belonging to different populations.
\newblock \emph{Journal of the Royal Statistical Society}, 109(3):296, 1946.

\bibitem{leonard2014}
C.~L\'eonard.
\newblock A survey of the Schr\"odinger problem and some of its connections
  with optimal transport.
\newblock \emph{Discrete and Continuous Dynamical Systems},
  34(4):1533--1574, 2014.

\bibitem{otto2001}
F.~Otto.
\newblock The geometry of dissipative evolution equations: the porous medium
  equation.
\newblock \emph{Communications in Partial Differential Equations},
  26(1--2):101--174, 2001.

\bibitem{ambrosio_gigli_savare2008}
L.~Ambrosio, N.~Gigli, and G.~Savar\'e.
\newblock \emph{Gradient Flows in Metric Spaces and in the Space of Probability
  Measures}.
\newblock Lectures in Mathematics ETH Z\"urich. Birkh\"auser, Basel,
  2nd edition, 2008.

\bibitem{villani2009}
C.~Villani.
\newblock \emph{Optimal Transport: Old and New}.
\newblock Grundlehren der mathematischen Wissenschaften, vol.~338.
  Springer-Verlag, Berlin, 2009.

\end{thebibliography}
\end{document}